\documentclass[12pt]{amsart}
\usepackage{preamble_general}
\usepackage[
    colorlinks,
    linkcolor=blue,
    citecolor=blue,
    urlcolor=blue,
    linktocpage=true
]{hyperref}
\usepackage[nameinlink]{cleveref}
\usepackage{orcidlink}
\usepackage{bm}
\usepackage{preamble_abbreviations}

\title[height pairing for differential forms]{The Archimedean height pairing for differential forms on degeneration of Riemann surfaces}
\date{\today}

\author{Junyu Cao \orcidlink{0009-0003-6875-3981}}

\begin{document}
\begin{abstract} 
    We define the Archimedean height pairing for fiberwise 
    cohomologically trivial differential forms on a one-parameter 
    degeneration of Riemann surfaces, and we study its asymptotic 
    behavior. The proof relies on recent work by Dai--Yoshikawa on the 
    asymptotics of small eigenvalues. As an application, we relate 
    this pairing to the current-valued pairing of Filip--Tosatti, 
    extending their construction to broader geometric settings.
\end{abstract}
\maketitle
\setcounter{tocdepth}{1}
\tableofcontents
\section*{Introduction}
Let \(\pi\colon X \to S \simeq \D\) be a proper surjective holomorphic map from a
complex surface \(X\)
to a Riemann surface isomorphic to the unit disk. We assume that \(\pi\) has connected fibers and \(X_0 \coloneqq \pi^{-1}(0)\) is
the unique singular fiber. Assume that $X$ is equipped with a 
\Kahler metric.

For two smooth real $(1,1)$-forms $\alpha, \beta$ on $X$ that are 
cohomologically trivial on smooth fibers $X_s \coloneqq \pi^{-1}(s)
(s \neq 0)$, i.e., $\int_{X_s} \alpha = \int_{X_s}\beta = 0$, we define \emph{the Archimedean height pairing} between
them.
\begin{definition}[Archimedean height pairing, \Cref{dfn:pairing_of_forms}]
    For any $s \neq 0$, let $-\vphi_s \in C^\infty(X_s)$
    be a potential for $\alpha$ on $X_s$, i.e.,
    $\ddc \vphi_s + \alpha\mid_{X_s} = 0$. Then 
    \emph{the Archimedean height pairing} $\la \alpha, \beta \ra$ between $\alpha, \beta$
    is a function on smooth locus $S^\circ \coloneqq S \setminus \{0\}$:
    \[
    \la \alpha, \beta \ra(s) \coloneqq \int_{X_s} \vphi_s \beta,
    \quad s \neq 0.
    \]
    Since $\pi$ is a submersion on $S^\circ$, $\la \alpha, \beta \ra \in C^\infty(S^\circ)$.
\end{definition}

We study its asymptotics around the singular fiber of the degeneration.
\begin{theorem}[\Cref{thm:continuity_of_pairing_full}, \Cref{cor:continuity_of_pairing_general_fiber}]
    \label{thm:main_thm}
    There is a constant $c_{\alpha, 
    \beta} \in \R$ such that 
    \(
    \la \alpha, \beta \ra - c_{\alpha, \beta} \log\abs{s}^2
    \)
    extends continuously to $S$. In particular,
    $\la \alpha, \beta \ra \in L^\infty(S^\circ)$
    if and only if $\la \alpha, \beta \ra \in C^0(S)$.

    When $X_0$ is a reduced divisor and $C_1, \cdots, C_N$ are its irreducible components, the constant $c_{\alpha, \beta}$ is determined in the following
    way: Let $M$ be the intersection matrix of irreducible components
    of $X_0$, given by $(M_{ij})_{1 \leq 
    i, j \leq N} = (C_i \cdot C_j)_{1 \leq i, j \leq N}$, and
    let  
    $M^+$ be the 
    Moore-Penrose pseudoinverse of $M$. Then $c_{\alpha, \beta} = 
    \mathbf{v}_\alpha^T (M^+)^T  \mathbf{v}_\beta$, where 
    \[
    \mathbf{v}_\alpha = \left( \int_{C_1} \alpha, \dots, \int_{C_N} \alpha \right)^T, \quad 
    \mathbf{v}_\beta = \left( \int_{C_1} \beta, \dots, \int_{C_N} \beta \right)^T.
    \]
\end{theorem}

\begin{remark}
    The name \emph{Archimedean height pairing} comes from 
    Arakelov theory, e.g. in Bost's work \cite{Bost1990}. But their 
    pairings are defined for algebraic cycles/divisors. 
    
    Precisely, let $D, E$ be two divisors 
    on the one-parameter degeneration $\pi \colon X \to S$
    of algebraic curves. Suppose $D, E$ are of 
    relative degree $0$ and their supports do not intersect. For a divisor $D$, we 
    define its Green's current $g_{D_s}$ on general fibers $X_s$ following \cite{DH2015}*{Equation (2.1)}. The $g_{D_s}$ is a distribution on 
    $X_s$ solving the equation
    \[
    \pa \dbar g_{D_s} + \pi i \delta_{D_s} = 0,
    \]
    where $D_s$ is the restriction of $D$ on $X_s$.
    We remark that $g_{D_s}$ is smooth on $X_s$ away
    from the support of $D_s$.
    Then the Archimedean height pairing between 
    $D$ and $E$ is a well-defined function on $S^\circ$:
    \[
    s \mapsto \la D_s, E_s \ra_\infty \coloneqq 
    g_{D_s}([E_s]) = \sum_{i} a_i g_{D_s}(q_i),
    \]
    where $E_s = \sum_i a_i q_i$.

In Holmes--de Jong's paper \cite{DH2015}*{Theorem 2.1}, 
    they derive a similar asymptotic result for the Archimedean height pairing $\la D_s, E_s \ra_\infty$ on degenerations of algebraic 
    curves.
\end{remark}

\begin{remark}
    Similar asymptotic results were obtained by Yoshikawa \cite{Yoshikawa2006} and Eriksson--Freixas--Mourougane \cite{EFG2018} in their study of Quillen metrics around singular fibers. Analogous behavior for the canonical heights of elliptic curves around singular fibers was established by
    Silverman \cites{Silverman1,Silverman2,Silverman3} 
    and recently by 
    DeMarco--Mavraki \cite{DeMarcoMavraki2020}.
\end{remark}

The study of the Archimedean height pairing for forms is motivated 
by the following dynamical scenario. Let \(X\) be a projective 
elliptic K3 
surface with a parabolic automorphism \(T\) preserving the 
fibration:
\[
\begin{tikzcd}
    X \arrow[r, "T"]   \arrow[d, "\pi"']  & X
    \arrow[ld, "\pi"] \\
    B \coloneqq\mathbb{P}^1   &             
\end{tikzcd}
\]

In Filip--Tosatti's paper \cite{Filip-Tosatti21}*{Theorem 3.2.14},
they introduced the ``\emph{current-valued pairing}" $\eta_B(T, [\alpha])$, which is a current on $B$, for $\alpha$ a smooth closed real $(1, 1)$-form on $X$ that is cohomologically trivial on smooth fibers.

The \emph{current-valued pairing} can be used to characterize the
limit properties of large iterations of the automorphism $T$. 
For example, they proved in \cite{Filip-Tosatti21}*{Corollary 3.2.20} that 
\begin{equation}
    \label{eqn:intro_limit_current}
    \lim_{n \to \infty} \frac{T^n_* \omega}{n^2} = \frac{1}{2}
    \pi^*\eta_B(T, [T_*\omega - \omega]),
\end{equation}
where $\omega$ is a \Kahler form on $X$.

When the singular fibers of $\pi$ are reduced and irreducible,
Filip--Tosatti established the continuity of potentials
of $\eta_B$ in \cite{Filip-Tosatti21}*{Proposition 3.2.11}.
In this paper, we remove their assumption on the singular fibers and 
prove the continuity in general.
\begin{proposition}[\Cref{prop:continuity_of_current_val_pairing}]
    \label[proposition]{prop:intro_continuity_cur_pairing}
    Let $\mathrm{Aut}_\pi(X)$ be the group of automorphisms
    of $X$ preserving the fibration. 
    There exists a finite index subgroup $G$ of $\mathrm{Aut}_\pi(X)$ such that the \emph{current-valued pairing} $\eta_B$
    restricted on $G$ has a continuous potential.
\end{proposition}
This proposition is proved by the following observation, 
which relates \emph{the Archimedean height pairing} and 
the potential of \emph{current-valued pairing}. 

We use 
$u$ to denote a potential of the \emph{current-valued pairing} $\eta_B(T, [\alpha])$,
then we have (in \Cref{eqn:relation_height_pairing_and_current_pairing})
\[
    u(s)  =
    \int_{X_s} f \omega + \la \alpha, (T^* - I)\omega\ra,
\]
where $f \in C^\infty(X)$ and $\omega$ is a normalized 
\Kahler form on $X$. The first term $s \mapsto \int_{X_s} f \omega$ is well-studied by Barlet (\Cref{thm:fib_integral_barlet}) and is Hölder continuous. Therefore, \Cref{prop:intro_continuity_cur_pairing} follows from 
the continuity of \emph{the Archimedean height pairing}, which arises as the second term of the potential $u$.

Using \Cref{eqn:intro_limit_current} (\cite{Filip-Tosatti21}*{Corollary 3.2.20}), we have the following
result on parabolic dynamics of elliptic K3 surfaces.
\begin{corollary}[\Cref{cor:dynamical_consequences}]
    Let $G$ be the finite index subgroup of $\mathrm{Aut}_\pi(X)$ defined in \Cref{prop:intro_continuity_cur_pairing}. Then, for any \Kahler form $\omega$ on $X$ and any $T \in G$, the limit current 
    $\eta \coloneqq \lim_{n \to \infty} \frac{T^n_* \omega}{n^2}$ has 
    a continuous potential.

    Furthermore, we prove that the convergence of currents $\frac{(T^n)_* \omega}{n^2} \to \frac{1}{2}\pi^*\eta_B(T, [T_*\omega - \omega])$ cannot be in $C^0_{\rm loc}(X\setminus X_0)$.
    Therefore, the sequence $\left\{\frac{T^n_* \omega}{n^2}\right\}_{n \geq 1}$ provides a counterexample to Tosatti's question \cite{tosatti2025}*{Question 1.5} when we let $\omega$ be a Ricci-flat metric.
\end{corollary}

\begin{remark}
    The regularity of limit currents is also discussed in Cantat--Dujardin's recent paper \cite{Cantat-Dujardin-2023}*{Theorem D}.
    
    If the automorphism is hyperbolic, its associated limit current will generally have a Hölder-continuous potential. However, it is unknown whether the potential of a limit current given by a parabolic automorphism is Hölder-continuous.
\end{remark}

\begin{remark}
    Other related works on the parabolic dynamics of elliptic surfaces with general singular fibers include Duistermaat's monograph \cite{Duistermaat2010} and Cantat--Dujardin's recent paper \cite{Cantat-Dujardin-2023-B}*{Section 3}.
\end{remark}

Here we sketch the proof of \Cref{thm:main_thm}.
By a standard reduction, we can assume that $X_0$ is reduced and
$\alpha, \beta$
have a zero integral on each irreducible component of $X_0$.
Then it remains to show the continuity of the pairing in 
\Cref{thm:continuity_of_pairing_zero_singular_integral}:
We introduce \emph{the preferred potentials} 
(\Cref{dfn:preferred_potential}) for $\alpha,
\beta$. Then by Dai--Yoshikawa's work \cite{Dai-Yoshikawa2025},
we are able to derive some estimates on preferred potentials
(\Cref{thm:boundedness_of_preferred_potential_general}, \Cref{thm:preferred_potential_integrability_singular_fiber}).
Together with some continuity results of fiber integral
around singular fibers (\Cref{prop:central_convergence_compactness},
\Cref{cor:fiber_int_continuity_central_convergence}), 
we derive the continuity of the pairing (which is a fiber integral) in \Cref{thm:continuity_of_pairing_zero_singular_integral}.

This article is organized as follows. In \Cref{section:1-technical-preparation}, we collect and prove some technical results. In \Cref{section:2-spectral-geometry-dai-yoshikawa}, we study 
the properties of eigenfunctions with small eigenvalues
on smooth fibers of a degeneration. We use Dai-Yoshikawa's 
work \cite{Dai-Yoshikawa2025} heavily in this section. 
In \Cref{section:3-preferred-potentials}, we introduce the 
notion of a \emph{preferred potential} and establish its asymptotic 
estimates based on \Cref{section:2-spectral-geometry-dai-yoshikawa}.
In \Cref{section:4-Archimedean-pairing}, we introduce the 
\emph{Archimedean height pairing} and prove its asymptotics
around singular fibers. In \Cref{section:5-k3-dynamics},
we apply previous results to study the parabolic dynamics of 
K3 surfaces. In \Cref{section:6-appendix}, we prove 
a technical lemma on the asymptotics of a fiber integral.

\subsection*{Notation and conventions}  
On a complex manifold $X$, we define 
${\di}^c = \frac{i}{4\pi}(\bar{\partial} - \partial)$, 
so $\ddc = \frac{i}{2\pi}\partial\bar{\partial}$.

For a closed smooth differential form $\alpha$ on a compact differential manifold,
we use $[\alpha]$ to denote its de Rham cohomology class.

We use $\D \coloneqq \{z \in \C \colon\ \abs{z} < 1\}$ to denote the unit disk in complex plane, and $\D^\circ \coloneqq \D\setminus \{0\}$ 
the punctured unit disk. For a positive number $r > 0$, we also 
define $\D_r \coloneqq \{z \in \C \colon\ \abs{z} < r\}$ and 
$\D_r^\circ \coloneqq \D_r \setminus \{0\}$.

\vspace{1em}
\noindent\textbf{Acknowledgements.} The author thanks his advisor Valentino Tosatti
for introducing the topic, for helpful discussions
and for his constant support. The author also thanks
Simion Filip, Mikhael Gromov, Fulin Xu for helpful discussions,
and Xianzhe Dai, Romain Dujardin, Robin de Jong, Ken-Ichi Yoshikawa for helpful conversations, as well as careful reading and feedback on an earlier draft.

\section{Technical Preparation}
\label{section:1-technical-preparation}
We collect and prove some technical results that will be used later.
\subsection{Laplacian on singular varieties}
Let $X$ be an algebraic variety with the induced Fubini-Study metric, and let $X_{\rm reg}$ be the regular part of $X$.

Following Li-Tian's work \cite{Li_Tian_Heat_Kernel_1995}*{Section 4}, we define the ordinary Laplacian $\Delta$
on compactly supported smooth functions on $X_{\rm reg}$ and 
let $\overline{\Delta}$ be its Friedrichs extension. Then we have

\begin{theorem}[\cite{Li_Tian_Heat_Kernel_1995}*{Section 4}]
    \label{thm:harmonic_function_algebraic_variety}
    The Friedrichs extension $\overline{\Delta}$ defined 
    on functions is self-adjoint on a compact variety $X$ 
    (without boundary) whose singular set $\Sing X$ is at least
    of real codimension $2$.

When $f \in W^{1,2}(X_{\rm reg})$ is \emph{harmonic}, i.e., 
    $\Delta f = 0$ almost everywhere, we have $\di f = 0$ and $f$ is locally constant.
\end{theorem}

\begin{remark}
    This theorem holds when $X$ is a \Kahler space, see 
    \cite{GPSS-Sobolev}*{Lemma 11.2}.
\end{remark}

\subsection{Continuity of fiber integral à la Barlet}
In \cite{Barlet}, Barlet derived the asymptotic expansion of the 
fiber integral for the degeneration of complex manifolds near singular fibers.
\begin{theorem}[\cite{Barlet}*{Théorème 1}]
\label{thm:fib_integral_barlet}
Let $D = \{s \in \mathbb{C} \mid |s| < 1\}$, and let $X$ be a reduced and irreducible complex analytic space of dimension $n + 1$. Let $\pi : X \to D$ be a surjective holomorphic map, and let us denote by $X_s$ the cycle $\pi^{-1}(s)$. Let $K$ be a compact subset of $X$. There exist rational numbers $r_1 \leq \dots \leq r_k$ in $[0, 2) \cap \mathbb{Q}$ such that for every $C^\infty$ differential form $\vphi$ of type $(n, n)$ on $X$ with support in $K$, the function $F_\vphi \coloneqq \int_{X_s} \vphi$ admits, as $s \to 0$, an asymptotic expansion of the form:
\[
F_\vphi(s) \sim \sum_{\substack{r=r_1,\dots,r_k \\ j=0,\dots,n \\ (m,m')\in\mathbb{N}^2}} T^{r,j}_{m,m'}(\vphi) s^m \overline{s}^{m'} |s|^r (\log|s|)^j
\]
where $T^{r,j}_{m,m'}$ is a $(1, 1)$-current on $X$. 
\end{theorem}

Since $F_{\vphi}(s)$ does not blow up as $s \to 0$, we have 
$T^{0,j}_{0,0}(\vphi) = 0$ for $j \geq 1$.
Let $0 < r$ be the minimum of non-zero $r_1
\leq \cdots \leq r_k$ in $[0, 2) \cap \Q$,
then for $0 < a < \min(r, 1)$, we have 
\[\abs{F_\vphi(s) - F_\vphi(0)} \leq C(\vphi) \abs{s}^a.\]
Thus $F_{\vphi}(s)$ is Hölder continuous around $0$.

\begin{remark}
    When the base of the fibration $\pi$ is of higher 
    dimension, Takayama extends Barlet's result in his 
    recent paper \cite{Takayama2020}.
\end{remark}

\subsection{Convergence on the singular fiber}

Let $\pi \colon X \to \D$ be a proper surjective holomorphic map 
from a complex manifold $X$ of complex dimension 
$p = q + 1$ to the unit disk.
Suppose that $\pi$ has connected fiber and $X_0 \coloneqq \pi^{-1}(0)$ 
is the unique singular fiber. We define
$\Sing(X_0) \coloneqq \{ x \in X_0 \mid d\pi_x = 0 \}$ and 
$X_{0, \rm{reg}} \coloneqq X_0 \setminus \Sing(X_0)$. 
Let $X_s \coloneqq \pi^{-1}(s)$.

We introduce the following frequently
used notion of convergence on the singular fiber.
\begin{definition}[Convergence on the singular fiber]
    \label[definition]{dfn:central_fib_convergence}
    Under the above settings, let $(s_n)_{n \geq 1} \subset \D^\circ$ be a sequence such that $s_n \to 0$ as $n \to \infty$.
    Let $(f_{s_n})_{n \geq 1}$ be 
    a family of $C^{k}$ functions on $X_{s_n}$ and $f$
    be a $C^k$ function on $X_{0,\rm reg}$. Around any 
    $p \in X_{0,\rm reg}$, the map $\pi$ is a locally trivial fibration 
    by Ehresmann's theorem, so we have an open neighborhood $U$ of $p$
    with the following diagram:
\[\begin{tikzcd}
	{i: U} & {V \times \D_{\epsilon}} \\
	& {\D_{\epsilon}}
	\arrow["\simeq", from=1-1, to=1-2]
	\arrow["\pi"', from=1-1, to=2-2]
	\arrow[from=1-2, to=2-2]
\end{tikzcd}\]

Then we have a family of diffeomorphisms $i_{s} 
\colon U \cap X_s \to V \times \{s\}$. For $s \in \D_{\epsilon}$,
we define $\tilde{f_s}\coloneqq (i_s)_{\ast}(f_s)$ 
to be a $C^k$
function on $V$, and define $\tilde{f}\coloneqq (i_0)_{\ast}(f)$. 

We say \emph{$f_{s_n}$ converges to $f$ in $C^k_{\rm loc}$
on $X_{0,\rm reg}$} if around any $p \in X_{0,\rm reg}$,
we have a coordinate chart $U \simeq V \times \D_\epsilon$ and 
\(
    \tilde{f_{s_n}} \to \tilde{f} 
\)
in the $C^{k}_{\rm loc}(V)$ topology. This notion of convergence 
does not depend 
on the choice of coordinate charts.
\end{definition}

\begin{remark}
    Let $f \colon X \setminus X_0 \to \R$ be a continuous
    function. Then $f$ can be extended
    to be a continuous function on $X \setminus \Sing X_0$
    if and only if $f_{s_n}$ converges to a continuous
    function $f_0 \in C^0(X_{0,\rm reg})$ in $C^0_{\rm loc}$ on $X_{0,\rm reg}$ for any sequence $s_n \to 0$.
\end{remark}

We shall show how the convergence on the singular fiber 
could imply the convergence of the fiber integral in 
the following.
\begin{lemma}
    \label[lemma]{lem:central_fib_integrability}
    Suppose $f_{s_n} \to f$ in $C^0_{\rm loc}$ on $X_{0,\rm reg}$ and 
    $f$ is 
    a continuous function on $X_{0,\rm reg}$. Let $\omega$ be a 
    Hermitian form on $X$. Then 
    \[
        \int_{X_{0,\rm reg}} \abs{f}^l \omega^q \leq 
        \limsup_{n \to \infty} \int_{X_{s_n}} \abs{f_{s_n}}^l \omega^q
    \]
    for $1 \leq l < \infty$.

    Generally, if $f_{s_n} \to f$ in $C^k_{\rm loc}$ on $X_{0,\rm reg}$, then 
    \[
        \int_{X_{0,\rm reg}} \abs{\nabla^r f}^l \omega^q \leq 
        \limsup_{n \to \infty} \int_{X_{s_n}} \abs{\nabla^r f_{s_n}}^l 
        \omega^q
    \]
    for $0 \leq r \leq k, 1 \leq l < \infty$.
\end{lemma}
Before the proof, we need a partition of unity 
tailored for our situation:
\begin{lemma}
    \label[lemma]{lem:partition_of_unity_compact_set}
    Let $X$ be a smooth manifold and $K \subset X$ be a compact subset. Suppose $\{U_i\}_{i=1}^N$ is a finite open covering of $K$. Then there exists a collection of smooth functions $\{\rho_i\}_{i=1}^N$ with $\rho_i: X \to \mathbb{R}$ such that:
\begin{enumerate}
    \item $\operatorname{supp}(\rho_i) \subset U_i$ for each $i = 1, \dots, N$.
    \item $0 \leq \rho_i(x) \leq 1$ for all $x \in X$.
    \item $\sum_{i=1}^N \rho_i(x) = 1$ for all $x \in K$.
    \item $\sum_{i=1}^N \rho_i(x) \leq 1$ for all $x \in X$.
\end{enumerate}
\end{lemma}
\begin{proof}(of \Cref{lem:partition_of_unity_compact_set})
    Since $K$ is compact and $\{U_i\}_{i=1}^N$ is an open covering of $K$, the collection
    of open sets $\mathcal{U} = \{U_1, 
    U_2, \cdots, U_N, X \setminus K\}$ is an open 
    covering of $X$. Let $\{\rho_1, 
    \rho_2, \cdots, \rho_{N}, \rho_{N+1}\}$ 
    be a partition of unity subordinate to 
    $\mathcal{U}$. 
    Then we obtain the desired $\{\rho_i\}_{i=1}^N$.
\end{proof}

\begin{proof}(of \Cref{lem:central_fib_integrability})
    Let $U \simeq V \times \D$ be a coordinate chart around some point 
    in $X_{0, \rm reg}$ trivializing the fibration $\pi \colon X 
    \to \D$. Let $h$
    be a continuous function on $X$ compactly supported on $U$. Then 
    if $f_{s_n} \to f$ in $C^k_{\rm loc}$ on $X_{0, \rm reg}$, 
    we have $h_{s_n} \abs{\nabla^r f_{s_n}} \to h_0 \abs{\nabla^r f}$
    in the $C^{0}(V)$-topology. In particular, we have 
    $\int_{X_{0, \rm reg}}h \abs{\nabla^r f}^l \omega^q = \lim_{n \to \infty} \int_{X_{s_n}} 
    h\abs{\nabla^r f_{s_n}}^l 
    \omega^q$.

For any compact subset $K$ of $X_{0, \rm reg}$, let $\{U_i\}_{i=1}^N$ 
    be an open covering of $K$ by open sets in $X\setminus \Sing X_0$ and 
    we assume that each 
    $U_i \simeq V_i \times \D$ is a coordinate chart
    that trivializes the fibration $\pi$.

    We then take the partition of unity $\{h_i\}_{i=1}^N$
    subordinate to $\{U_i\}_{i=1}^N$ as in 
    \Cref{lem:partition_of_unity_compact_set}, then for each $i$, we have 
    \[
    \int_{X_{s_n}} 
    h_i\abs{\nabla^r f_{s_n}}^l 
    \omega^q
    \to 
    \int_{X_{0, \rm reg}}h_i \abs{\nabla^r f}^l \omega^q.
    \]
    Therefore, we have 
    \begin{align*}
        \int_{K}\abs{\nabla^r f}^l \omega^q
    \leq
    \sum_{i=1}^N
    \int_{X_{0, \rm reg}}h_i \abs{\nabla^r f}^l \omega^q
    &= \lim_{n \to \infty}
    \sum_{i = 1}^N 
    \int_{X_{s_n}} 
    h_i\abs{\nabla^r f_{s_n}}^l 
    \omega^q
    \\
    &\leq 
    \limsup_{n \to \infty} \int_{X_{s_n}} \abs{\nabla^r f_{s_n}}^l 
    \omega^q.
    \end{align*}
    As $K$ can be taken to be any compact
    subset of $X_{0, \rm reg}$, we conclude 
    that 
    \[
    \int_{X_{0, \rm reg}} \abs{\nabla^r f}^l \omega^q \leq 
    \limsup_{n \to \infty} \int_{X_{s_n}} \abs{\nabla^r f_{s_n}}^l 
    \omega^q.
    \]
\end{proof}

If $f_{s_n}$ in the previous lemma (\Cref{lem:central_fib_integrability}) is zero in a tubular neighborhood 
of $\Sing X_0$, then we have the equality for the limit of fiber 
integral.
\begin{lemma} 
    \label[lemma]{lem:fiber_int_continu_no_singularity}
    Suppose $f_{s_n}$
    converges to $f$ in $C^0_{\rm loc}$ on $X_{0,\rm reg}$ and there 
    exists $\epsilon > 0$ such that $f_{s_n}, f$ are zero in 
    a tubular neighborhood
    $B(\Sing X_0, \epsilon) \coloneqq 
    \{p \in X \colon \dist(p, \Sing X_0) < \epsilon\}$ of 
    $\Sing X_0$. Then for any 
    smooth $(q, q)$ form $\alpha$, we have 
    \[
        \int_{X_{s_n}} f_{s_n} \alpha \to 
        \int_{X_0} f \alpha, 
        \quad 
        \text{as } n \to \infty.
    \]
\end{lemma}
\begin{proof}
Let $U_\epsilon = B(\Sing X_0, \epsilon)$ be an open tubular neighborhood of the singular locus. 
Since $X_0 \setminus U_\epsilon$ is compact and contained in 
the regular locus of $X_0$, we can cover it by a 
finite number of open sets $\{U_i\}_{i = 1}^N$ in $X\setminus \Sing X_0$ and 
we assume that each 
$U_i \simeq V_i \times \D$ is a coordinate chart
that trivializes the fibration $\pi$. Then we have an open
covering $U_\epsilon \cup \{U_i\}_{i=1}^N$ of $X_0$.

Since $\pi$ is proper, we can find $\epsilon_0 > 0$ such that 
$\pi^{-1}(\overline{\D_{\epsilon_0}}) \subset U_\epsilon \cup
\cup_{i=1}^N U_i$. We apply \Cref{lem:partition_of_unity_compact_set}
to the compact set $\pi^{-1}(\overline{\D_{\epsilon_0}})$ and its 
open covering $U_\epsilon \cup \{U_i\}_{i=1}^N$, then we have a 
collection of smooth functions $\{\rho_i \colon U_i \to \R\}_{i=0}^N$ (we denote 
$U_\epsilon$ by $U_0$)
such that 
\begin{itemize}
    \item $\operatorname{supp}(\rho_i) \subset U_i$ for each $i =0, 1, \dots, N$,
    \item $0 \leq \rho_i(x) \leq 1$ for all $x \in X$,
    \item $\sum_{i=0}^N \rho_i(x) = 1$ for all $x \in \pi^{-1}(\overline{\D_{\epsilon_0}})$. 
\end{itemize}

Then for $n \gg 1$ such that $\abs{s_{n}} < \epsilon_0$, we 
have 
\[
\int_{X_{s_n}} f_{s_n} \alpha = 
\sum_{i=0}^N \int_{X_{s_n}} \rho_i f_{s_n} \alpha 
= \sum_{i=0}^N \int_{X_{s_n} \cap U_i} \rho_i f_{s_n} \alpha
\]
And 
\[
\int_{X_{0}} f \alpha = 
\sum_{i=0}^N \int_{X_{0}} \rho_i f \alpha 
= \sum_{i=0}^N \int_{X_{0} \cap U_i} \rho_i f \alpha
\]

Since $f_{s_n}, f$ are zero in $U_0 = U_\epsilon$, we have 
$$\int_{X_{0} \cap U_i} \rho_0 f \alpha = 0 = 
\int_{X_{s_n} \cap U_0} \rho_0 f_{s_n} \alpha$$

For $1 \leq i \leq N$, the fibration $\pi$ on $U_i \simeq V_i \times \D$
is trivial and $(\rho_i f_{s_i}\alpha) \mid_{V_i} \to (\rho_i f \alpha) \mid_{V_i}$ in $C^0_{\rm loc}$. So 
\[
\int_{X_{s_n} \cap U_i} \rho_i f_{s_n} \alpha 
= \int_{V_i}(\rho_i f_{s_i}\alpha) \mid_{V_i} 
\to \int_{V_i} (\rho_i f \alpha) \mid_{V_i} = \int_{X_{0} \cap U_i} \rho_i f \alpha.
\]
In summary, we have 
\[
\int_{X_{s_n}} f_{s_n} \alpha \to 
        \int_{X_0} f \alpha, 
        \quad 
        \text{as } n \to \infty.
\]
\end{proof}

We can extend above lemma if the contribution of fiber integral around singularities is small
enough, which is the following lemma.
\begin{lemma}
\label[lemma]{lem:fib_inte_continuity_full}
Suppose $f_{s_n}$
converges to $f$ in $C^0_{\rm loc}$ on $X_{0,\rm reg}$ and $X_0$
is reduced. Let $\omega$ be a \Kahler metric on $X$.

For any  $\epsilon > 0$, we define 
\[
I(\epsilon) = 
\limsup_{n\to \infty} \int_{B(\Sing X_0, \epsilon) \cap X_s}
\abs{f_{s_n}} \omega^q,
\]
where $B(\Sing X_0, \epsilon) \coloneqq 
\{p \in X \colon \dist(p, \Sing X_0) < \epsilon\}$ is a tubular 
neighborhood of $\Sing X_0$.

If $I(\epsilon) \to 0$ as $\epsilon \to 0$ and $f$ is $L^1$-integrable on $X_0$, then for any 
smooth $(q, q)$ form $\alpha$, we have
\[
\int_{X_{s_n}} f_{s_n} \alpha \to
\int_{X_0} f \alpha,
\quad
\text{as } n \to \infty.
\]
\end{lemma}
\begin{proof}
    For any fixed $\epsilon > 0$, we have a cutoff function
    $\eta_\epsilon \in C^\infty(X)$ such that 
    \begin{itemize}
        \item $0 \leq \eta_\epsilon \leq  1$,
        \item $\eta_\epsilon = 1$ on $B(\Sing X_0, \epsilon/2)$,
        \item $\eta_\epsilon = 0$ on $X\setminus B(\Sing X_0, \epsilon)$.
    \end{itemize}
    Then by \Cref{lem:fiber_int_continu_no_singularity}, we have 
    \[
    \int_{X_{s_n}} f_{s_n} (1 - \eta_\epsilon) \alpha \to
    \int_{X_0} f(1 - \eta_\epsilon) \alpha,
    \quad
    \text{as } n \to \infty.
    \]
    Therefore, we have 
    \[
    \limsup_{n \to \infty} \abs{\int_{X_{s_n}} f_{s_n}\alpha
    - \int_{X_0} f \alpha}
    \leq \limsup_{n \to \infty} \abs{\int_{X_{s_n}}f_{s_n}\eta_\epsilon \alpha}
    + \abs{\int_{X_0} f \eta_\epsilon \alpha}.
    \]

    For the first term, we have 
    \[
    \abs{\int_{X_{s_n}}f_{s_n}\eta_\epsilon \alpha}
    \leq \norm{\frac{\alpha}{\omega^q}}_{L^\infty}
    \int_{B(\Sing X_0, \epsilon)} \abs{f_{s_n}} \omega^q.
    \]
    So 
    \[
    \limsup_{n\to \infty} \abs{\int_{X_{s_n}}f_{s_n}\eta_\epsilon \alpha}
    \leq  \norm{\frac{\alpha}{\omega^q}}_{L^\infty} I(\eps).
    \]
    We then have 
    \begin{equation}
        \label{eqn:temp_lem_fib_int}
        \limsup_{n \to \infty} \abs{\int_{X_{s_n}} f_{s_n}\alpha
        - \int_{X_0} f \alpha}
        \leq \norm{\frac{\alpha}{\omega^q}}_{L^\infty} I(\eps)
        + \abs{\int_{X_0} f \eta_\epsilon \alpha}.
    \end{equation}

Note that
    \[
    \abs{\int_{X_0} f \eta_\epsilon \alpha} 
    \leq \norm{\frac{\alpha}{\omega^q}}_{L^\infty} \int_{B(\Sing X_0, \epsilon) \cap X_0} \abs{f} \omega^q
    \omega^q
    \]
    and $\vol(B(\Sing X_0, \epsilon) \cap X_0) \to 0$ as $\epsilon \to 0$
    by \Cref{lem:uniform_tubular_volume}. So $\abs{\int_{X_0} f \eta_\epsilon \alpha}
    \to 0$ as $\epsilon \to 0$ by $L^1$-integrability of $f \alpha$.

    Letting $\epsilon \to 0$ in \Cref{eqn:temp_lem_fib_int}, we obtain 
    \[\limsup_{n \to \infty} \abs{\int_{X_{s_n}} f_{s_n}\alpha
    - \int_{X_0} f \alpha} = 0.\]

    We then conclude the result.
\end{proof}

We shall see that $I(\epsilon) \to 0$ is natural under certain 
uniform integrability conditions.
\begin{lemma}
\label[lemma]{lem:fib_integral_converge_with_integrability}
Suppose $f_{s_n}$
converges to $f$ in $C^0_{\rm loc}$ on $X_{0,\rm reg}$ and $X_0$
is reduced. Let $\omega$ be a \Kahler metric on $X$.
If there is $1 < r \leq \infty$ and a uniform constant $L < \infty$ such that $\norm{f_{s_n}}_{L^r(X_{s_n}, \omega_{s_n})} < L$, then for any smooth $(q, q)$ form $\alpha$, we have
\[
\int_{X_{s_n}} f_{s_n} \alpha \to
\int_{X_0} f \alpha,
\quad
\text{as } n \to \infty.
\]
\end{lemma}
\begin{proof}
    By \Cref{lem:fib_inte_continuity_full}, it suffices to check 
    the following two conditions
    \begin{itemize}
        \item The contribution of fiber integral around singularities
        $$I(\epsilon) = \limsup_{n\to \infty} \int_{B(\Sing X_0, \epsilon) \cap X_s}
\abs{f_{s_n}} \omega^q$$ goes to zero as $\epsilon \to 0$.
        \item
        The $f$ is $L^1$-integrable on $X_0$.
    \end{itemize}

    The first condition is justified by the Hölder inequality
    and the volume estimate of tubular neighborhood in  \Cref{lem:uniform_tubular_volume}
    \begin{align*}
        \limsup_{n\to \infty} \int_{B(\Sing X_0, \epsilon) \cap X_s}
        \abs{f_{s_n}} \omega^q
        &\leq \limsup_{n\to \infty}
        \left( \int_{X_{s_n}} \abs{f_{s_n}}^r \omega^q \right)^{1/r} 
        \left( \int_{B(\Sing X_0, \epsilon) \cap X_{s_n}}\omega^q \right)^{1/p}\\
        &\leq L \limsup_{n\to \infty} \vol(B(\Sing X_0, \epsilon) \cap X_{s_n})\\
        &\leq L C \epsilon^2 \quad (\text{by \Cref{lem:uniform_tubular_volume}}).
    \end{align*}

    The second condition follows from \Cref{lem:central_fib_integrability}, 
    we have
    \[
    \norm{f}_{L^r(X_0)} \leq \limsup_{n \to \infty} \norm{f_{s_n}}_{L^r}
    < L
    \]
    for $1 < r \leq \infty$. So $f \in L^1(X_0)$.
\end{proof}

We prove the uniform bound for volume of tubular neighborhoods
of singular set, which is used in previous two lemmas and is 
standard:
\begin{lemma}[Volume bound for singular tubular neighborhoods]
    \label[lemma]{lem:uniform_tubular_volume}
    Let $\pi \colon X \to \D$ be a proper surjective holomorphic map 
    from a complex manifold $X$ of complex dimension 
    $p = q + 1$ to the unit disk.
    Suppose that $\pi$ has connected fiber and $X_0 \coloneqq \pi^{-1}(0)$ 
    is the unique singular fiber. Suppose $X_0$
    is reduced. Then $\Sing X_0$ is of complex codimension $\geq 1$
    in $X_0$.

    Let $\omega$ be a Kähler metric on $X$, and denote by $B(\Sing X_0, \epsilon)$ the $\epsilon$-tubular neighborhood of $\Sing X_0$ in $X$ with respect to the distance induced by $\omega$. Then there exist constants $C > 0$ and $\epsilon_0 > 0$ such that for all $s \in \D$ near $0$ (including $s = 0$) and all $0 < \epsilon < \epsilon_0$, we have
    \[
        \int_{B(S, \epsilon) \cap X_s} \omega^q \leq C \epsilon^{2}.
    \]
\end{lemma}

\begin{proof}
    
    Since $\Sing X_0$ is a compact complex analytic
    subvariety of $X_0$ of complex dimension less than $q - 1$, 
    its $2(q - 1)$-dimensional Hausdorff measure is finite. 
    Thus, there exists a constant $C_1 > 0$ such that for any sufficiently 
    small $\epsilon > 0$, the compact set $\Sing X_0$ can be covered by a finite 
    number of balls $N(\epsilon)$ of radius $\epsilon$ centered at points 
    $x_i \in \Sing X_0$, where 
    \[
        N(\epsilon) \leq C_1 \epsilon^{-2(q - 1)}.
    \]
    The $\epsilon$-tubular neighborhood $B(\Sing X_0, \epsilon)$ is strictly contained in the union of the concentric balls $B(x_i, 2\epsilon)$.

    Next, we bound the volume of $X_s$ within these local balls. For any $s$ (including $s=0$), $X_s$ is a purely $q$-dimensional complex analytic cycle. Because $\omega$ is a Kähler metric on $X$, the standard Bishop-Lelong theorem (or monotonicity formulas) guarantees that the volume ratio is monotonically increasing. Thus, we take $R_0 > 2 \epsilon_0$ such that for all $x \in S$ and all $2\epsilon < 2 \epsilon_0 < R_0$, we have
    \[
        \frac{1}{(2\epsilon)^{2q}} \int_{B(x, 2\epsilon) \cap X_s} \omega^q \leq \frac{1}{R_0^{2q}} \int_{B(x, R_0) \cap X_s} \omega^q.
    \]
    Furthermore, because $d\omega = 0$ and the fibers $X_s$ are homologous in $X$, the total volume of the fibers $V = \int_{X_s} \omega^q$ is a topological constant independent of $s$. We have a uniform constant
    $C_2$ and the following inequality
    \[
        \int_{B(x, 2\epsilon) \cap X_s} \omega^q \leq \left( \frac{V}{R_0^{2q}} \right) (2\epsilon)^{2q} \coloneqq C_2 \epsilon^{2q}.
    \]

    Finally, combining the covering estimate and the local volume bound, we obtain, by the subadditivity of the volume,
    \[
        \int_{B(\Sing X_0, \epsilon) \cap X_s} \omega^q \leq \sum_{i=1}^{N(\epsilon)} \int_{B(x_i, 2\epsilon) \cap X_s} \omega^q \leq N(\epsilon) \cdot C_2 \epsilon^{2q} \leq C_1 C_2 \epsilon^{2}.
    \]
    Setting $C = C_1 C_2$, the bound holds uniformly for all $s$ near $0$, completing the proof.
\end{proof}

Finally, we give a sufficient condition on the existence of limit
of convergence on the singular fiber:
\begin{proposition}
    \label[proposition]{prop:central_convergence_compactness}
    We work under the above settings and assume that $X_0$ is reduced.
    Let $\vphi$ be a real
    function on $X \setminus X_0$. We assume that 
    $\vphi|_{X_s}$ is smooth on all regular fibers $X_s$
    for $s \neq 0$, and define a family of smooth $(1, 1)$
    forms $\beta_s \coloneqq \ddc (\vphi|_{X_s})$ on each $X_s$
    for $s \neq 0$. 
    
    Let $\omega$ be a \Kahler form on $X$. If 
    the following conditions hold:
    \begin{itemize}
        \item Around any point of $X_{0, \rm reg}$, we have an open
        neighborhood $U$ such that the $L^\infty$-norms $\norm{\tr_{\omega_s} \beta_s}_{L^\infty(X_s \cap U)}$ are uniformly bounded for $\abs{s} \ll 1$.
        \item The family of functions $f_{s} \coloneqq \tr_{\omega_{s}} \beta_{s}$ converges in $C^0_{\rm loc}$ to a continuous function $f_0$ on $X_{0,\rm reg}$.
        \item The $L^2$-norms $\norm{\vphi_s}_{L^2(X_s, \omega_s)}$
        and $\norm{\di \vphi_s}_{L^2(X_s,\omega_s)}$ are uniformly 
        bounded for $\abs{s} \ll 1$.
    \end{itemize}
    Then for any sequence $(s_n)_{n \geq 1} \subset \D^\circ$ such that  
    $s_n \to 0$ and any $0 < a < 1$, there exists a subsequence of $s_n$ 
    (by abusing the notation, we denote the subsequence by $s_n$ itself) 
    such that
    \begin{itemize}
        \item $\vphi_{s_n}$ converges
        in $C^{1, a}_{\rm loc}$
        to a $C^{1, a}$ function $\vphi_0$
        on $X_{0,\rm reg}$,
        \item $\vphi_0 \in W^{2, p}_{\rm loc}(X_{0, \rm reg})$ for any $1 < p < \infty$, and $\vphi_0$ solves $-\Delta_{\omega_0} \vphi_0 = f_0$ almost everywhere on $X_{0,\rm reg}$,
        \item $\vphi_0 \in W^{1, 2}(X_{0, \rm reg})$.
        \item If $\alpha$ is a smooth $(q, q)$-form, then 
        \[
        \int_{X_{s_n}} \vphi_{s_n} \alpha \to \int_{X_0} \vphi_0 \alpha
        \]
    \end{itemize}

    Furthermore, for two different limits $\vphi_0, \vphi_0'$, their
    difference $\vphi_0 - \vphi_0'$ is locally constant on 
    $X_{0, \rm reg}$.
\end{proposition}

\begin{proof}
    Around any point of $X_{0, \rm reg}$, we have an open neighborhood 
    $U$ such that the 
    $L^\infty$-norms $\norm{\tr_{\omega_s} \beta_s}_{L^\infty(X_s \cap U)}$
    are uniformly bounded. By shrinking $U$,
    we assume $U \simeq V \times \D_\epsilon$ that 
    trivializes the fibration
    $\pi$.
    
    Around any point of $V$, we take a relatively compact 
    coordinate ball $W$. Then $\vphi_s$ solves 
    \[
        -\Delta_{\omega_s} \vphi_s = \tr_{\omega_s} \ddc_z 
        \vphi_s = \tr_{\omega_s} \beta_s \quad \text{on } W.
    \]
    Note that $\Delta_{\omega_s}$ is a family of uniformly 
    elliptic operators for $\abs{s} \ll 1$. 
    By standard local $L^\infty$ bounds for 
    elliptic PDEs (see Gilbarg--Trudinger 
    \cite{GT-PDE}*{Theorem 8.24}), 
    we can shrink $W$ to a slightly smaller domain 
    $\tilde{W} \Subset W$ to obtain
    \begin{align*}
        \norm{\vphi_s}_{L^\infty(\tilde{W})} &\leq C 
        \left( \norm{\vphi_s}_{L^2(W)} + 
        \norm{\tr_{\omega_s} \beta_s}_{L^\infty(W)} \right) \\
        &\leq C \left( \norm{\vphi_s}_{L^2(X_s)} + 
        \norm{\tr_{\omega_s} \beta_s}_{L^\infty(X_s\cap U) } \right),
    \end{align*}
    where $C$ is a uniform constant independent of $s$. By our assumptions, the right-hand side is uniformly bounded in $s$.

    Next, by interior $L^p$ estimates for elliptic equations 
    (see \cite{GT-PDE}*{Theorem 9.11}), for any $1 < p < \infty$,
     we can further shrink $\tilde{W}$ to $W'$ to obtain
    \begin{align*}
        \norm{\vphi_s}_{W^{2, p}(W')} &\leq C_p 
        \left( \norm{\vphi_s}_{L^\infty(\tilde{W})} + 
        \norm{\tr_{\omega_s} \beta_s}_{L^p(\tilde{W})} \right) \\
        &\leq C_p' \left( \norm{\vphi_s}_{L^\infty(\tilde{W})} + 
        \norm{\tr_{\omega_s} \beta_s}_{L^\infty(X_s\cap U)} \right).
    \end{align*}
    Thus, $\norm{\vphi_s}_{W^{2, p}(W')}$ is uniformly bounded for 
    any $p > 1$. 

    By the Sobolev embedding theorem (Morrey's inequality), 
    for $p > 2q$ (where $q = \dim_\C X_s$), 
    $W^{2, p}(W')$ embeds compactly into 
    $C^{1, b}(W')$ for $0 < b < 1 - 2q/p$. 
    Since $p$ can be chosen arbitrarily large, 
    $\norm{\vphi_s}_{C^{1, b}(W')}$ is uniformly 
    bounded for any $0 < b < 1$.

    To patch these local estimates together, 
    note that $X_{0,\rm reg}$ is second-countable, 
    so we can choose a countable exhaustion by 
    compact subsets $K_1 \Subset K_2 \Subset \cdots \Subset X_{0,\rm reg}$
    such that $\cup_m K_m = X_{0,\rm reg}$. 
    For each $K_m$, we can cover it by finitely 
    many such charts $W'$. By the Ascoli-Arzel\`a 
    theorem and diagonal argument, we can 
    extract a single subsequence $(s_n)$ such that 
    $\vphi_{s_n}$ converges in $C^{1, a}_{\rm loc}$ to 
    a limit function $\vphi_0 \in C^{1, a}(X_{0,\rm reg})$ 
    for any given $a \in (0, 1)$. Furthermore, the uniform bounds in 
    $W^{2,p}_{\rm loc}$ guarantee that 
    $\vphi_0 \in W^{2,p}_{\rm loc}(X_{0,\rm reg})$ and that 
    $\vphi_{s_n} \to \vphi_0$ weakly in $W^{2,p}_{\rm loc}$.

    Since $f_{s_n} \coloneqq \tr_{\omega_{s_n}} \beta_{s_n}$ 
    converges in $C^0_{\rm loc}$ to the continuous limit $f_0$, 
    passing to the limit in the PDE implies that $\vphi_0$ 
    satisfies $-\Delta_{\omega_0} \vphi_0 = f_0$ almost 
    everywhere on $X_{0,\rm reg}$. 
    
    By the third condition and 
    \Cref{lem:central_fib_integrability}, 
    we have $\vphi_0 \in W^{1, 2}(X_{0, \rm reg})$.

    Since $\norm{\vphi_s}_{L^2(X_s, \omega_s)}$ is uniformly bounded, 
    by \Cref{lem:fib_integral_converge_with_integrability} we have 
    \[
    \int_{X_{s_n}} \vphi_{s_n} \alpha \to \int_{X_0} \vphi_0 \alpha
    \]
    for any smooth $(q, q)$-form $\alpha$ on $X$.

    Suppose $\vphi_0$ and $\vphi_0'$ are two limits arising 
    from different subsequences of $\vphi_{s_n}$. 
    Both solve the same equation 
    $-\Delta_{\omega_0} \vphi = f_0$ almost everywhere 
    on $X_{0,\rm reg}$. 
    So their difference 
    $h \coloneqq \vphi_0 - \vphi_0' \in W^{2,p}_{\rm loc}(X_{0,\rm reg})$ 
    is a weakly harmonic function on $X_{0,\rm reg}$ 
    with respect to $\omega_0$. Since $h \in W^{1, 2}(X_{0, \rm reg})$ 
    globally and the singular locus of $X_0$ is of real codimension at 
    least $2$, by \Cref{thm:harmonic_function_algebraic_variety}, 
    $h$ is locally constant.
\end{proof}

As a corollary, we have the following result on continuity of fiber 
integral
\begin{corollary}
    \label[corollary]{cor:fiber_int_continuity_central_convergence}
    We work under the assumption and notations of \Cref{prop:central_convergence_compactness}. Assume 
    that all conditions in \Cref{prop:central_convergence_compactness}
    are satisfied. 

    If $\alpha$ is a smooth $(q, q)$-form such that $\int_{C} \alpha = 0$ for each irreducible component $C$ of $X_0$, then the fiber integral $I(s) = \int_{X_s} \vphi_s \alpha$ is continuous at $s = 0$.
\end{corollary}
\begin{proof}
    We prove it by contradiction. Otherwise, 
    we would have two sequences $\{s_n\}$ and 
    $\{s_n'\}$ with $s_n, s_n' \to 0$ such that 
    \[
    \lim_{n \to\infty} I(s_n) \neq \lim_{n \to\infty} I(s_n').
    \]

    By \Cref{prop:central_convergence_compactness}, 
    we assume that $\vphi_{s_n} \to \vphi_0$, 
    $\vphi_{s_n'} \to \vphi_0'$ in $C^{1, a}_{\rm loc}$ on 
    $X_{0, \rm reg}$, then 
    \[
    \lim_{n \to\infty} I(s_n) = \int_{X_0} \vphi_0 \alpha,
    \quad \lim_{n \to\infty} I(s_n') = \int_{X_0} \vphi_0' \alpha.
    \]

    Since $\vphi_0 - \vphi_0'$ is locally constant on 
    $X_{0, \rm reg}$ and each irreducible component $C$
    of $X_0$ is connected, $\vphi_0 - \vphi_0'$ is piecewise constant
    on irreducible components of $X_0$. Therefore,
    \[
    \int_{X_0}(\vphi_0 - \vphi_0') \alpha
    = \sum_{C} \mathrm{constant }\int_{C} \alpha = 0,
    \]
    where $C$ runs over all irreducible components of $X_0$. 
    
    Thus, $\lim_{n \to\infty} I(s_n) - \lim_{n \to\infty} I(s_n') = 0$, which is a contradiction.
\end{proof}

\subsection{Smoothness of spectral projection}
Let $\pi\colon 
\mathfrak{X} \to S$ be a proper, smooth submersion between smooth
manifolds. In Demailly's expository article \cite{demailly1996}, he proved the following:
\begin{lemma}[\cite{demailly1996}*{Lemma 10.3}]
    \label[lemma]{lem:proj_of_laplace_is_smooth}
    Let $g$ be a Riemannian metric on $\mathfrak{X}$, and let $g_s$
    be its restriction on each $X_s = \pi^{-1}(s)$. Then $-\Delta_s
    \colon C^\infty(X_s, \R) \to C^{\infty}(X_s, \R)$
    is a smooth family of elliptic operators. Furthermore, the eigenvalues
    of $-\Delta_s$ 
    \[
        0 < \lambda_0(s) \leq \lambda_1(s) \leq \cdots
    \]
    are continuous in $s$.

    Moreover, if $\lambda$ is not in the spectrum
    $\{\lambda_k(s_0)\}_{k \geq 0}$ of 
    $-\Delta_{s_0}$, the direct sum $W_{s, \lambda}
    \subset C^\infty(X_t)$ of 
    eigenspaces of $-\Delta_s$ with 
    eigenvalues $\leq \lambda$ defines a $C^\infty$
    vector bundle (i.e., varies smoothly), in a 
    neighborhood of $s_0$.
\end{lemma}

As a direct corollary, we have
\begin{corollary}
    \label[corollary]{cor:projection_is_smooth}
    Let 
    $u$ be a smooth function on $\mathfrak{X}$
    and $P_{\leq \lambda}^s \colon C^{\infty}(X_s)
    \to C^\infty(X_s)$ be the orthogonal projection
    onto $W_{s, \lambda}$.

    Suppose $\lambda$ is not in the spectrum
    $\{\lambda_k(s_0)\}_{k \geq 0}$ of 
    $-\Delta_{s_0}$, then $P_{\leq \lambda}^s(u)$ and $(1 - P_{\leq \lambda}^s)(u)
    \eqqcolon P_{> \lambda}^{s}(u)$ are smooth 
    on $\pi^{-1}(U)$, where $U \subset S$ is an 
    open neighborhood of $s_0$.
\end{corollary}

\section{Spectral geometry of degeneration of algebraic curves}
\label{section:2-spectral-geometry-dai-yoshikawa}
In this section, we study the spectral geometry of the degeneration of algebraic curves, closely following the recent work of Dai--Yoshikawa \cite{Dai-Yoshikawa2025}. 
In the first two subsections (\Cref{section:small-eigenvalues},
\Cref{section:thick-thin-model-functions}), we recall the tools developed in 
\cite{Dai-Yoshikawa2025}. In the last 
subsection \Cref{section:small-eigenfunctions}, we study the properties of 
eigenfunctions with small eigenvalues.

We introduce the setup for this section.
\begin{setup}
\label[setup]{setup:degeneration_with_reduced_fiber}
Let \(\pi\colon X \to S \simeq \D\) be a proper surjective holomorphic map from a 
complex surface \(X\)
to a Riemann surface isomorphic to the unit disk. We assume that \(\pi\) has connected fibers and \(X_0 = \pi^{-1}(0)\) is the unique singular fiber. We define
$\Sing(X_0) \coloneqq \{ x \in X_0 \mid d\pi_x = 0 \}$, 
$X_{0, \rm{reg}} \coloneqq X_0 \setminus \Sing(X_0)$ and 
$X^\circ \coloneqq X \setminus X_0$. 
Let $X_s \coloneqq \pi^{-1}(s)$.

Assume that $X$ is equipped with a \Kahler form $\omega$, and let the corresponding metric be $g^X$. 
Let \(g_s = g^{X}|_{X_s}\), then \((X_s, g_s)\) (\(s\neq 0 \)) is 
a compact Riemann surface with a \Kahler metric. We assume 
that $\vol(X_s, g_s) = \int_{X_s} \omega_s = 1$.

Assume that \(X_0\) is a reduced and reducible divisor of \(X\). In particular, \(\pi\) has
 only isolated critical points on \(X_0\).
\end{setup}

\subsection{Properties of small eigenvalues}
\label{section:small-eigenvalues}
For \(s\neq 0\), let \(0 < \lambda_1(s) \leq \lambda_2(s) \leq \cdots\) be the 
eigenvalues of the Laplacian (with respect to \(g_s\))
\(\square_s = \bar{\pa}^\ast\bar{\pa}\) counted with multiplicities. 
For \(s = 0\), we regard \(\square_0\) as the Friedrichs extension of the 
Laplacian on \(X_{0, \reg} = X_0\backslash\Sing X_0\) with compact support. 
Then by Yoshikawa's paper \cite{Yoshikawa1997}, \(\lambda_i(s)\) is a continuous function in \(s \in S\) for each \(i \geq 1\).

Let \(N = \sharp \{\text{irreducible components of }X_0\} = \dim \ker(\square_0)\), we have a precise control for small eigenvalues around \(0\):
\btheorem[{\cite{Dai-Yoshikawa2025}*{Theorem 0.2}}] 
\label{thm:eigenvalue_growth_rate}
    There exist constants \(C_0, C_1 > 0\) such that for all \(s \in \D^\circ\), 
    \[
        \frac{C_0}{\log(\abs{s}^{-1})} \leq \lambda_1(s) \leq \cdots \leq \lambda_{N-1}(s) \leq \frac{C_1}{\log(\abs{s}^{-1})}
    \]
\etheorem
\vone
In contrast to the first $N-1$ small eigenvalues, the remaining eigenvalues have a lower bound as $s$ approaches $0$.
\btheorem[{\cite{Dai-Yoshikawa2025}*{Theorem 3.8}}]
\label{thm:high_eigenvalues_lowe_bound}
For all $k \geq 1$, the function $\lambda_k(s)$ on $S^\circ$
extends to a continuous function on $S$. Moreover, for all $k \geq N$, 
$\lambda_k(s) \geq \lambda$ for a positive number $\lambda > 0$.
\etheorem

We call $\lambda_1(s), \cdots, \lambda_{N-1}(s)$ \emph{small 
eigenvalues}. Using \Cref{cor:projection_is_smooth} and the spectral 
gap of eigenvalue $\lambda_{N}(s)$ in \Cref{thm:high_eigenvalues_lowe_bound},
we define the spectral projection onto the direct sum of
eigenspaces with small eigenvalues.
\begin{lemma}
    \label[lemma]{lem:proj_onto_small_freq_is_smooth}
    We take a constant $0 < \delta_0 < 1$ depending
    only on the geometry of the degeneration and work on $\D_{\delta_0}^\circ$.

    For $s \in \D_{\delta_0}^\circ$, 
    let $W_{s, \rm low} \subset 
    C^{\infty}(X_s)$ be the space spanned by eigenfunctions
    with small eigenvalues 
    $\lambda_1(s), \cdots, \lambda_{N-1}(s)$. Let 
    $P_{\rm low}^s \colon C^{\infty}(X_s) \to 
    W_{s, \rm low}$ be the orthogonal projection onto 
    $W_{s, \rm low}$. 
    
    Then for any smooth function $u$ on $X \setminus X_0$,
    its low frequency part $P_{\rm low}^s(u)$
    is smooth on $\pi^{-1}(\D_{\delta_0}^\circ)$. Similarly,
    its high frequency part $(1 - P_{\rm low}^s)(u)$ is 
    also smooth on $\pi^{-1}(\D_{\delta_0}^\circ)$. 
\end{lemma}
\begin{proof}(of \Cref{lem:proj_onto_small_freq_is_smooth})
    Using the notation in \Cref{thm:eigenvalue_growth_rate}
    and \Cref{thm:high_eigenvalues_lowe_bound}, we define 
    $\delta_0$ by
    \[
    \D_{\delta_0}^\circ = \{s\in S^\circ \colon 
    \frac{C_1}{\log(\abs{s}^{-1})} \leq \frac{1}{2}
    \lambda\}
    \]
    On the punctured disk $\D_{\delta_0}^\circ$, we have 
    \[
    \lambda_{N-1}(s) \leq \frac{C_1}{\log(\abs{s}^{-1})}
    \leq \frac{1}{2}
    \lambda < \lambda_N(s).
    \]

    For $s \in \D_{\delta_0}^\circ$, we have $P_{\rm low}^s
    = P^s_{\leq \frac{3}{4}\lambda}$. By \Cref{cor:projection_is_smooth},
    we know $P^s_{\leq \frac{3}{4}\lambda}(u)$ is smooth 
    on $\pi^{-1}(\D_{\delta_0}^\circ)$.
\end{proof}
Throughout this paper, we  
denote $P_{\rm low}^s(u)$ by $u_{\rm low}$ and 
$(1 - P_{\rm low}^s)(u)$ by $u_{\rm high}$.

\subsection{Thick-thin decomposition of nearby fibers}
\label{section:thick-thin-model-functions}
From a geometric perspective, each nearby smooth fiber $X_s$ ($s \neq 0$) admits a thick-thin decomposition. As the family degenerates, the thick part of $X_s$ converges to the regular locus of the singular fiber, $X_0 \setminus \text{Sing}(X_0)$, while the thin part collapses to the singular points $\text{Sing}(X_0)$. Since $X_0 \setminus \text{Sing}(X_0)$ consists of $N$ connected components, the thick part of $X_s$ correspondingly decomposes into $N$ disjoint regions. 

Following \cite{Dai-Yoshikawa2025}*{Section 6}, for each $1 \leq i \leq N$, we construct a \emph{model function} $\Upsilon^{(i)}_s \in C^\infty(X_s)$ in \Cref{prop:good_test_on_X_s} that serves as a smooth approximation of the characteristic function on the $i$-th thick component. 

We construct these \emph{model functions} in the following three
steps:
\vone

\noindent \hypertarget{step1}{\textbf{Step 1:}}   We construct a family of diffeomorphisms which sends
the smooth part of the singular fiber to nearby smooth fibers. 
The idea is to pick a \(C^\infty\) complex vector field \(v\) on \(X \backslash \mathrm{Crit}(f)\) satisfying \(f_* v = \pa/\pa s\), and use this vector field to flow points on the singular fiber to nearby smooth fibers. 

We collect useful results in \cite{Dai-Yoshikawa2025}*{Section 6} and omit the construction. Then we have 
the following theorem
\begin{theorem}
\label{thm:dai-yoshikawa-package}
For $s \in \D$ and $\abs{s} \ll 1$. Let $\nu \in \Z_{\geq 1}$
be a positive integer depending on the degeneration. Let $\epsilon(s) = 2 \abs{s}^{\frac{1}{8\nu}}$, we have a family
of diffeomorphisms 
\begin{align*}
    F_s \colon 
    X_0 \setminus \cup_{p \in \Sing X_0} B(p, \epsilon(s)) 
    &\to F_s(X_0 \setminus \cup_{p \in \Sing X_0} B(p, \epsilon(s))) (\subset X_s)\\
    z & \mapsto F_s(z) \in X_s
\end{align*}
The family of diffeomorphisms enjoys the following properties:
\begin{enumerate}
    \item \cite{Dai-Yoshikawa2025}*{Equation (6.2)} When $s = 0$, then $F_0 \colon 
    X_0 \to X_0$ is the identity map.
    \item \cite{Dai-Yoshikawa2025}*{Equation (6.2)} The family $F_s$ is smooth in $s$.
    \item \cite{Dai-Yoshikawa2025}*{Equation (6.3)} For $z \in X_0$
    in the domain of $F_s$, we have 
    \[
        \dist(F_s(z),  z) \leq K \abs{s}^{\frac{3}{4}} 
    \]
    \item \cite{Dai-Yoshikawa2025}*{Lemma 6.4} Let $\omega$ be the \Kahler form in our settings, then 
    \[
        \norm{(F_s)^\ast \omega_s - \omega_0}
        _{L^\infty(X_0 \setminus \cup_{p \in \Sing X_0} B(p, \epsilon(s)) )} 
        \leq K_2 \abs{s}^{\frac{1}{2}}.
    \]
    \item \cite{Dai-Yoshikawa2025}*{Lemma 6.5} For any two functions $\chi, \chi' \in C^\infty_0(X_0 \setminus \cup_{p \in \Sing X_0} B(p, \epsilon(s)))$, we have 
    \begin{align*}
        \abs{((F_s)_\ast \chi, (F_s)_\ast \chi')_{L^2(X_s)}  -(\chi, \chi ')_{L^2(X_0)}} &\leq 
        K_3 \abs{s}^{\frac{1}{2}} \norm{\chi}_{L^2(X_0)}
        \norm{\chi'}_{L^2(X_0)},\\
        \abs{\norm{\di\bigl((F_s)_\ast \chi\bigr)}_{L^2(X_s)}
        - \norm{\di \chi}_{L^2(X_0)}} 
        &\leq K_3 \abs{s}^{\frac{1}{2}} \norm{\di \chi}_{L^2(X_0)}.
    \end{align*}
\end{enumerate}
\begin{remark}
    \label[remark]{rmk:of_dai_yoshikawa_package}
    The result in \Cref{thm:dai-yoshikawa-package}.(4) 
    can be generalized to any differential form on $X$ with 
    $C^{1}$ regularity. For any smooth form $\alpha$ on $X$
    with $C^1$-coefficients, we have 
    \[
        \norm{(F_s)^\ast \alpha_s - \alpha_0}
        _{L^\infty(X_0 \setminus \cup_{p \in \Sing X_0} B(p, \epsilon(s)) )} 
        \leq K_2 \abs{s}^{\frac{1}{2}}.
    \]
    The proof follows verbatim as in \cite{Dai-Yoshikawa2025}*{Lemma 6.4}. 
\end{remark}

\end{theorem}

We estimate the volume of the subset in $X_s$
that is not in the image of $F_s$.
\begin{corollary}
    \label{cor:measure_of_unmapped}
    Let $O_s$ be the image of $F_s$ inside $X_s$. Then 
    $O_s$ is open in $X_s$ since $F_s$ is a 
    diffeomorphism. And for $\abs{s} \ll 1$, there 
    is a uniform constant $K_4 > 0$ such that
    \[
        \int_{X_s \setminus O_s} \omega_s \leq K_4 
        \abs{s}^{\frac{1}{4 \nu}}.
    \]
\end{corollary}
\begin{proof}
    Let $\chi_s \in C^\infty_0(X_0 \setminus \cup_{p \in \Sing X_0} B(p, \epsilon(s)))$ be a test function such that 
    \begin{itemize}
        \item We have $0 \leq \chi_s \leq 1$.
        \item On $X_0 \setminus \cup_{p \in \Sing X_0} B(p, 2\epsilon(s))$, $\chi_s \equiv 1$.
    \end{itemize}
    Then we have 
    \[
        0 \leq \vol(X_0) - (\chi_s, \chi_s)_{L^2(X_0)} 
        \leq \sum_{p \in \Sing X_0} K_4' (\epsilon(s))^2 = K_4'\abs{s}^{\frac{1}{4 \nu}} ,
    \]
    where $K_4'$ is a constant depending on the geometry of 
    singular fiber.

    Since $X_0$ is reduced and $\omega$ is \Kahler, we have $\vol(X_0) = \vol(X_s)$. Together with \Cref{thm:dai-yoshikawa-package}.(5), we have 
    \[
        0 \leq \vol(X_s) - ((F_s)_\ast \chi_s, (F_s)_\ast \chi_s)_{L^2(X_s)} 
        \leq K_3 \abs{s}^{\frac{1}{2}} + K_4'\abs{s}^{\frac{1}{4 \nu}}.
    \]

    Note that $0 \leq (F_s)_\ast \chi \leq 1$ and 
    $(F_s)_\ast \chi \in C^\infty_0(O_s)$, we have 
    \begin{align*}
        \int_{X_s \setminus O_s} \omega_s \leq 
        \vol(X_s) - ((F_s)_\ast \chi, (F_s)_\ast \chi)_{L^2(X_s)} \leq K_4 \abs{s}^{\frac{1}{4 \nu}}.
    \end{align*}
\end{proof}

\vone
\noindent \hypertarget{step2}{\textbf{Step 2:}}  We construct good test functions
on the singular fiber, which will be flowed using diffeomorphisms in \Cref{thm:dai-yoshikawa-package} to nearby fibers in \hyperlink{step3}{\textbf{Step 3}} for constructing
\emph{model functions}.

Let \(X_0 = C_1 + \cdots + C_N\) be the irreducible decomposition 
of the singular fiber $X_0$. 
For each singular point \(p \in \Sing X_0 \cap C_i\), we fix a local coordinate system 
\((U_p, \zeta), \zeta = (\zeta_1, \zeta_2)\in \C^2\) and \(\zeta(p) = 0\). 
We define \(r_p(z)\coloneqq\norm{\zeta(z)} = \sqrt{\abs{\zeta_1}^2 + \abs{\zeta_2}^2}\).

Then we construct a family of functions on \(X_0\).
\begin{lemma}[\cite{Dai-Yoshikawa2025}*{Lemma 6.6}]
    \label[lemma]{lem:good_test_central}
    For every \(0 < \epsilon \ll 1\), there exists a family of smooth functions \(\chi^{(i)}_\epsilon
    (1 \leq i \leq N)\), such that
    \begin{enumerate}
        \item For $1 \leq i \leq 
        N$, we have $\chi^{(i)}_\epsilon \in C^\infty_0(C_i \setminus \Sing(X_0))$
        \item We have \(0 \leq \chi^{(i)}_\epsilon \leq 1\). 
        On \(C_i \backslash \bigcup_{p \in \Sing X_0 \cap C_i} U_p\), 
        we have \(\chi^{(i)}_\epsilon = 1\).
        \item For any \(p \in \Sing X_0 \cap C_i\), 
        we have \(\chi^{(i)}_\epsilon(z) = 0\) if 
        \(r_p(z) \leq \frac{1}{2}\epsilon\) and 
        \(\chi^{(i)}_\epsilon(z) = 1\) if 
        \(r_p(z) \geq 2\sqrt{\epsilon}\).
        \item We have \(\norm{\di \chi^{(i)}_\epsilon}^2_{L^2} \leq K_4/(\log \epsilon^{-1})\) where $K_4>0$ is a constant depending only on $\pi \colon X \to S$ around $0 \in S$.
        \item For each $i$, the $\chi^{(i)}_\epsilon$ depends on 
        $\epsilon$ continuously.
    \end{enumerate}
\end{lemma}

\vone
\noindent \hypertarget{step3}{\textbf{Step 3:}} We use the family of diffeomorphisms
generated by flow in \hyperlink{step1}{\textbf{Step 1}} and the collection of 
good test functions on $X_0$ in \hyperlink{step2}{\textbf{Step 2}} to construct
the collection of \emph{model functions} on $X_s$.
\begin{proposition}[\cite{Dai-Yoshikawa2025}*{Section 6.2}]
    \label[proposition]{prop:good_test_on_X_s}
    For $0 < \abs{s} \ll 1$, there exists a family of smooth 
    functions $\Upsilon^{(i)}_s \in C^{\infty}(X_s)$ ($1 \leq i 
    \leq N$), such that
    \begin{enumerate}
        \item For any $1 \leq i \leq N$ and any $s$, there is a constant $C > 0$ depending only on the geometry of degeneration such that $0 \leq \Upsilon^{(i)}_s \leq C$.

        \item 
        For $1 \leq i \leq N$, the function $\Upsilon^{(i)}(s, z)
        \coloneqq \Upsilon^{(i)}_s(z)$ for $0 < \abs{s} \ll 1, z \in X_s$
        is continuous on $X \setminus X_0$, i.e., the family varies
        continuously in $s$ away from singular fibers.
        \item For $1 \leq i, j \leq N$, $i \neq j$, we have 
        $\supp \Upsilon^{(i)}_s \cap \supp \Upsilon^{(j)}_s
        = \emptyset$.
        \item For $1 \leq i \leq N$, we have estimates 
        \[
        \norm{\Upsilon^{(i)}_s}^2_{L^2(X_s)} = 1 + O(\abs{s}^{\frac{1}{8\nu}}),
        \quad 
        \norm{\di \Upsilon^{(i)}_s}^2_{L^2(X_s)} \leq \frac{K}{\log(\abs{s}^{-1})}.
        \]
        The constant $K > 0$ depends only on $(\pi \colon X \to \D, \omega)$.
        \item The function $\Upsilon_s^{(i)}(z)$ is constant for 
        $z \in X_s$ in  $F_s(X_0 \setminus \bigcup_{p \in \Sing X_0}
        B(p, 4 \epsilon(s)^{\frac{1}{2}}))$, where $\epsilon(s) = 2 \abs{s}^{\frac{1}{8\nu}}$.
    \end{enumerate}
\end{proposition}
\begin{proof}
    We define 
    \begin{equation}
        \label{eqn:test_fun_defn_upsilon}
        \Upsilon^{(i)}(s) = 
        (F_s)_{\ast}(\chi^{(i)}_{4\epsilon(s)})/ \sqrt{\mathrm{Area}(C_i)} \in C^{\infty}(X_s),
    \end{equation}
    where \(\epsilon(s) = 2 \abs{s}^{\frac{1}{8\nu}}\).

    Since $0 \leq \chi^{(i)}_{\epsilon(s)} \leq 1$, 
    \Cref{prop:good_test_on_X_s}.(1) is proved.
    
    Since $\chi^{(i)}_{\epsilon(s)}$ depends on $s$ continuously
    by \Cref{lem:good_test_central}.(5) and $F_s$ is continuous
    in $s$, $\Upsilon^{(i)}(s)$ varies continuously in $s$. Then \Cref{prop:good_test_on_X_s}.(2) is proved.

    For the remaining parts, we refer to  \cite{Dai-Yoshikawa2025}*{Section 6.2}.
\end{proof}

\begin{corollary}
    \label[corollary]{cor:local_constant_of_upsilon}
    For any $p \in X_0 \setminus \Sing X_0$, there is an open 
    neighborhood $U$ of $p$ such that $\Upsilon^{(i)}(s, z)
    \coloneqq \Upsilon^{(i)}_s(z)$ is constant for $1 \leq i \leq N$.

    Moreover, we can extend $\Upsilon^{(i)}(s, z)$ to be a continuous
    function on $X \setminus \Sing X_0$. And $\Upsilon^{(i)}(s, z)$
    is locally constant around $X_0 \setminus \Sing X_0$.
\end{corollary}
\begin{proof}
    By the definition of $F_s$ in \Cref{thm:dai-yoshikawa-package}, for any $p$ in $X_0 \setminus \Sing X_0$, we can take an open neighborhood $U$ of $p$
    in $X$ such that $U \cap X_s$ is contained in $F_s(X_0 
    \setminus \bigcup_{p \in \Sing X_0}
    B(p, 4 \epsilon(s)^{\frac{1}{2}}))$.
    From the definition of $\Upsilon^{(i)}$ in 
    \Cref{eqn:test_fun_defn_upsilon}, the $\Upsilon^{(i)}$
    equals to $1/\sqrt{\mathrm{Area}(C_i)}$ on $U$. So $\Upsilon^{(i)}$ is locally constant around any point in $X \setminus \Sing X_0$.

    Together with \Cref{prop:good_test_on_X_s}.(2), we know that 
    $\Upsilon^{(i)}$ is continuous on $X \setminus \Sing X_0$. Furthermore, 
    it equals to $1/\sqrt{\mathrm{Area}(C_i)}$ around any point 
    in  $X \setminus \Sing X_0$.
\end{proof}

We call $\Upsilon^{(i)}_s$ \emph{model functions}.

\subsection{Properties of small eigenfunctions}
\label{section:small-eigenfunctions}
Eigenfunctions associated with small eigenvalues are called \emph{small eigenfunctions}. We approximate these using the \emph{model functions} defined in \Cref{prop:good_test_on_X_s}. The results are summarized below.

\begin{theorem}
    \label{thm:approximate_eigenfunctions}
    Let $\{\Phi_i(s, z)\}_{i=1}^{N-1}$ be eigenfunctions of 
    $\Delta_{\omega_s}$ with eigenvalues $0 < \lambda_1(s) \leq \lambda_2(s) 
    \leq \cdots \leq \lambda_{N-1}(s)$, i.e.,
    \(
        -\Delta_{\omega_s} \Phi_i = \lambda_i(s) \Phi_i , 
    \)
    on $X_s$. Then for each $1 \leq j \leq N-1$, we have a decomposition
    for $\Phi_j$,
    \begin{equation}
        \label{eqn:eigenfunctions_correlation}
        \Phi_j(s, z) = \sum_{l=1}^{N}C_{l(j+1)}(s) \Upsilon^{(l)}(s, z)
        + R_j (s, z)
    \end{equation}
    for $0 < \abs{s} \ll 1$, where $C_{ij}(s) = \int_{X_s} \Upsilon^{(i)}_s \Phi_{j - 1}
    \omega_s = (\Upsilon^{(i)}_s, \Phi_{j - 1})_{L^2(X_s, \omega_s)}$. 
    And we have the following estimates
    \begin{enumerate}
        \item 
        There is a uniform constant $C$ such that $\abs{C_{lj}(s)}
        < C$ for all $1 \leq l \leq N$ and all $s$.
        \item There is a uniform constant $C$ such that 
        \[
            \norm{\di R_j}^2_{L^2(X_s)} \leq \frac{C^2}{\log(\abs{s}^{-1})},
            \quad
            \norm{R_j}^2_{L^2(X_s)} \leq C^2 \frac{1}{\log(\abs{s}^{-1})}.
        \]
        for all $1 \leq j \leq N$ and all $s$.
    \end{enumerate}
\end{theorem}
\begin{remark}
    We mention that the choice of small eigenfunctions is not 
    unique.
    Nevertheless, the non-uniqueness of choice can be controlled:
    Let $\{\Phi_i(s, z)\}_{i=1}^{N-1}$ and $\{\Phi'_i(s, z)\}_{i=1}^{N-1}$
    be two families of normalized small eigenfunctions. They span the same vector 
    space $V$ and are orthonormal bases under the $L^2$-inner 
    product. Thus, they are related by an orthogonal matrix, i.e., there exists a matrix $O \in O(N - 1)$ such that
    \[
    (\Phi'_1(s, z), \Phi'_2(s, z), \cdots, \Phi'_{N-1}(s, z))
    = (\Phi_1(s, z), \Phi_2(s, z), \cdots, \Phi_{N-1}(s, z)) O.
    \]
    From the equation \Cref{eqn:eigenfunctions_correlation}, we 
    have 
    \begin{align*}
    (\Phi'_1, \Phi'_2, \cdots, \Phi'_{N-1}) 
    &= (\Upsilon^{(1)}, \cdots, \Upsilon^{(N)}) (C_{l (j+1)})_{\substack{1 \leq l 
    \leq N\\ 1 \leq j \leq N-1}} O + (R_1, \cdots, R_{N-1})O\\
    &= (\Upsilon^{(1)}, \cdots, \Upsilon^{(N)}) (C'_{l (j+1)})_{\substack{1 \leq l 
    \leq N\\ 1 \leq j \leq N-1}} + (R_1', \cdots, R_{N-1}').
    \end{align*}

    So for another choice of eigenfunctions, the remainders
    are related by an orthogonal matrix and the estimates are the same 
    for $R_1', \cdots, R_{N-1}'$.
\end{remark}
\begin{proof}
    We begin with the Fourier expansion for each model function $\Upsilon^{(i)}_s$ $(1 \leq i \leq N)$
    on $X_s$ with respect to the lower frequency eigenfunctions
    $\{\Phi_j(s, z)\}_{j=1}^{N-1}$,
    \begin{equation}
    \label{eqn:Fourier_expansion_of_test}
        \Upsilon^{(i)}_s = 
        \underbrace{
            \int_{X_s} \Upsilon^{(i)}_s \omega_s + \sum_{j=1}^{N-1} c_{i j}(s) \Phi_j
        }_{\eqqcolon \Upsilon^{(i)}_{s, \mathrm{low}}} 
        + h^{(i)}_s,
    \end{equation}
    where $c_{ij}(s) = 
    \int_{X_s} \Upsilon^{(i)}_s {\Phi_j} \omega_s$ and $h^{(i)}_s$
    is the high frequency remainders. 
    
    We estimate $h^{(i)}_s$ in the following lemma:
    \blemma
    \label[lemma]{lem:high_freq_of_test}
    There is a uniform constant $K > 0$ depending only on the 
    degeneration such that 
    \[
    \norm{h^{(i)}_s}_{L^2(X_s, \omega_s)}^2 \leq 
    \frac{K}{\log(\abs{s}^{-1})}
    \]
    \elemma
    \begin{proof}(of \Cref{lem:high_freq_of_test}) Applying $-\Delta_{\omega_s}$ to \Cref{eqn:Fourier_expansion_of_test}, 
    we get the following equation on $X_s$
    \[
    -\Delta_{\omega_s} \Upsilon^{(i)}_s 
    = \sum_{j=1}^{N-1} c_{i j}(s) \lambda_j(s) \Phi_j - 
    \Delta_{\omega_s} h^{(i)}_s.
    \]
    
    Since $h^{(i)}_s$ has no low frequency modes, by spectral
    gap in \Cref{thm:high_eigenvalues_lowe_bound}, we have a 
    uniform $\lambda > 0$ such that
    \begin{equation}
    \label{eqn:tmp_high_estimates}
    \begin{split}
        \lambda \norm{h^{(i)}_s}_{L^2(X_s, \omega_s)}^2
        &\leq (-\Delta_{\omega_s} h^{(i)}_s, h^{(i)}_s)_{L^2(X_s)}\\
        &= -\left(\Delta_{\omega_s} \Upsilon^{(i)}_s + \sum_{j=1}^{N-1} c_{i j}(s) \lambda_j(s) \Phi_j, h^{(i)}_s\right)_{L^2(X_s)}\\
        &=\left(-\Delta_{\omega_s} \Upsilon^{(i)}_s, h^{(i)}_s\right)_{L^2(X_s)}\\
        &= \left(-\Delta_{\omega_s} \Upsilon^{(i)}_s, \Upsilon^{(i)}_s\right)_{L^2(X_s)} + \left(\Delta_{\omega_s} \Upsilon^{(i)}_s, \Upsilon^{(i)}_{s, \mathrm{low}}\right)_{L^2(X_s)}
    \end{split}
\end{equation}
    The equality in the third line holds since the high frequency term $h^{(i)}_s$ in the Fourier expansion (\Cref{eqn:Fourier_expansion_of_test}) is $L^2(X_s, \omega_s)$-orthogonal to 
    low frequency modes $\Phi_j$ ($1 \leq j \leq N-1$).

    Integrating by parts on the compact manifold $X_s$ and using the estimate for the model functions $\Upsilon^{(i)}_s$ (in \Cref{prop:good_test_on_X_s}(4)), we get
    \[
    \left(-\Delta_{\omega_s} \Upsilon^{(i)}_s, \Upsilon^{(i)}_s\right)_{L^2(X_s)} = \norm{\di \Upsilon^{(i)}_s}^2_{L^2(X_s)}
    \leq \frac{K}{\log(\abs{s}^{-1})},
    \]
    where $K$ is a uniform constant.

    Since $\Delta_{\omega_s}$ preserves orthogonality in 
    the Fourier expansion \Cref{eqn:Fourier_expansion_of_test},
    we have 
    
    \begin{align*}
        {\left(\Delta_{\omega_s} \Upsilon^{(i)}_s, \Upsilon^{(i)}_{s, \mathrm{low}}\right)_{L^2(X_s)}} &= {\left(\Delta_{\omega_s} \Upsilon^{(i)}_{s, \mathrm{low}}, \Upsilon^{(i)}_{s, \mathrm{low}}\right)_{L^2(X_s)}} \\
        &= - \norm{\di \Upsilon^{(i)}_{s, \mathrm{low}}}^2_{L^2(X_s)} \leq 0.
    \end{align*}

    Put things back into \Cref{eqn:tmp_high_estimates}, 
    we get 
    \[
    \lambda \norm{h^{(i)}_s}_{L^2(X_s, \omega_s)}^2
    \leq \frac{K}{\log(\abs{s}^{-1})}.
    \]
    So we can find a uniform $K_1 > 0$ such that for all $1 \leq i 
    \leq N$, we have 
    \[
    \norm{h^{(i)}_s}_{L^2(X_s, \omega_s)}^2 \leq 
    \frac{K'}{\log(\abs{s}^{-1})}.
    \]
    \end{proof}

    We now go back to our theorem.
    \begin{enumerate}
        \item 
        We use the notation $\Phi_0\coloneqq 1$ and consider 
        the following correlation matrix $C(s) = (C_{ij}(s))
        _{1 \leq i, j \leq N}$,
        $$
        C_{ij}(s) = \int_{X_s} \Upsilon^{(i)}_s \Phi_{j - 1}
        \omega_s = (\Upsilon^{(i)}_s, \Phi_{j - 1})_{L^2(X_s, \omega_s)}.
        $$
        Then $\abs{C_{ij}(s)} \leq \norm{\Upsilon^{(i)}_s}_{L^2} 
        \norm{\Phi_{j - 1}}_{L^2} = O(1)$. 
        \item 
        We then show that the correlation matrix $C(s)$ is 
        close to an orthogonal matrix with controlled error terms:
        \begin{align*}
            (CC^{T})_{ij} &= \sum_{l=1}^{N} C_{il}C_{jl}\\ 
            &= \sum_{l=1}^{N-1} 
            (\Upsilon^{(i)}_s, \Phi_{l - 1})_{L^2(X_s)}
            (\Upsilon^{(j)}_s, \Phi_{l - 1})_{L^2(X_s)} \\
            &= (\Upsilon^{(i)}_s, \sum_{l=1}^N\Phi_{l-1}  ( \Upsilon^{(j)}_s, \Phi_{l-1}) )
            \\
            &= (\Upsilon^{(i)}_s, \Upsilon^{(j)}_s  - h^{(j)}_s)\\
            &= \norm{\Upsilon^{(i)}_s}^2_{L^2(X_s)}\delta_{ij} - (\Upsilon^{(i)}_s, h^{(j)}_s)
        \end{align*}
        The last equality follows by \Cref{prop:good_test_on_X_s}(3).

        Since $\norm{\Upsilon^{(i)}_s}^2_{L^2(X_s)} = 
        1 + O(\abs{s}^{\frac{1}{8\nu}})$ (by \Cref{prop:good_test_on_X_s}(4)) and $\abs{(\Upsilon^{(i)}_s, h^{(j)}_s)} \leq \norm{h^{(j)}_s}_{L^2} \norm{\Upsilon^{(i)}_s}_{L^2} = O(\frac{1}{\log(s^{-1})})^{\frac{1}{2}}$ (by \Cref{lem:high_freq_of_test}), we have 
        \[
        C C^T = I +E,
        \]
        where $E$, the error term, is controlled by $\norm{E}
        \leq K_1 (\frac{1}{\log(s^{-1})})^{\frac{1}{2}}$. 
        So \(C^{-1} = C^T (I + E)^{-1}\). 
        
        Since $\norm{E} \ll 1$
        for $\abs{s} \ll 1$, we have 
        \[
        (I + E)^{-1} = I - E + E^2 - E^3 + \cdots \eqqcolon 
        I + E',
        \]
        where $E' = \sum_{i \geq 1}(-1)^{i}E^i$. We estimate that 
        \[
        \norm{E'} \leq \norm{E}(1 + \sum_{i \geq 1} \norm{E}^i) 
        \leq K_2 (\frac{1}{\log(s^{-1})})^{\frac{1}{2}}
        \]
        for some uniform constant $K_2$. Then 
        \[
        C^{-1} = C^T (I + E').
        \]

Note that 
    \begin{align*}
        \begin{pmatrix}
            \Upsilon^{(1)}_s - h^{(1)}_s\\
            \Upsilon^{(2)}_s - h^{(2)}_s\\
            \vdots\\
            \Upsilon^{(N)}_s - h^{(N)}_s
        \end{pmatrix}= 
        \begin{pmatrix}
            \Upsilon^{(1)}_{s, \mathrm{low}}\\
            \Upsilon^{(2)}_{s, \mathrm{low}}\\
            \vdots\\
            \Upsilon^{(N)}_{s, \mathrm{low}}
        \end{pmatrix}
         = C(s) 
        \begin{pmatrix}
            \Phi_0\\
            \Phi_1\\
            \vdots\\
            \Phi_{N-1}
        \end{pmatrix} = C(s) 
        \begin{pmatrix}
            1\\
            \Phi_1\\
            \vdots\\
            \Phi_{N-1}
        \end{pmatrix}
    \end{align*}

    We have 
    \begin{align*}
        \begin{pmatrix}
            1\\
            \Phi_1\\
            \vdots\\
            \Phi_{N-1}
        \end{pmatrix} = C^{-1}(s)\begin{pmatrix}
            \Upsilon^{(1)}_{s, \mathrm{low}}\\
            \Upsilon^{(2)}_{s, \mathrm{low}}\\
            \vdots\\
            \Upsilon^{(N)}_{s, \mathrm{low}}
        \end{pmatrix}
        = &C^{T}(s)(I + E')
        \begin{pmatrix}
            \Upsilon^{(1)}_{s, \mathrm{low}}\\
            \Upsilon^{(2)}_{s, \mathrm{low}}\\
            \vdots\\
            \Upsilon^{(N)}_{s, \mathrm{low}}
        \end{pmatrix}
    \end{align*}
        So for each \(1 \leq j \leq N - 1\), we have 
        \begin{align}\label{eqn:tmp_eigenfunctions_correlation}
            \Phi_j(s) &= \sum_{l=1}^{N}C_{l(j + 1)}(s)\Upsilon^{(l)}_{s, \mathrm{low}} + R_j.
        \end{align}
        And the remainder term $R_j = \left(C^T \cdot E'\cdot  (\Upsilon^{(1)}_{s, \mathrm{low}}, \cdots, \Upsilon^{(N)}_{s, \mathrm{low}})^T\right)_{j+1}$ is controlled by
        \[
        \abs{R_j} \leq \norm{C^T}_{L^\infty} \norm{E'}_{L^\infty} \norm{(\Upsilon^{(1)}_{s, \mathrm{low}}, \cdots, \Upsilon^{(N)}_{s, \mathrm{low}})^T}_{L^\infty} = O(\frac{1}{\log(s^{-1})})^{\frac{1}{2}}.
        \]
        So $\norm{R_j}_{L^2(X_s)} \leq C^2 \frac{1}{\log(\abs{s}^{-1})}$ for some constant $C > 0$.

        To estimate the derivative of the remainder, we take the 
        derivative of the \Cref{eqn:tmp_eigenfunctions_correlation}
        on $X_s$,
        \[
        \di \Phi_j(s) - C_{l(j + 1)}(s)\sum_{l=1}^{N}\di \Upsilon^{(l)}_{s, \mathrm{low}} = \di R_j.
        \]
        Then 
        \begin{align*}
            \norm{\di R_j}_{L^2}
            \leq \sum_{l=1}^N \norm{C_{l(j + 1)}}_{L^\infty}
            \norm{\di \Upsilon^{(l)}_{s, \mathrm{low}}}_{L^2}
            + \norm{\di \Phi_j(s)}_{L^2}.
        \end{align*}
        Since $\norm{\di \Phi_j(s)}_{L^2} = \lambda_{j}^{\frac{1}{2}}(s)$ and $\norm{\di \Upsilon^{(l)}_{s, \mathrm{low}}}^2_{L^2} = O\left(\frac{1}{\log(\abs{s}^{-1})}\right)$ by 
        \Cref{prop:good_test_on_X_s}(4), 
        we have 
        \[
        \norm{\di R_j}_{L^2}^2 \leq C^2 \frac{1}{\log(\abs{s}^{-1})}
        \]
        for some constant $C > 0$.
    \end{enumerate}
\end{proof}

\section{Estimates of preferred potentials}
\label{section:3-preferred-potentials}
The setup for this section is the same as \Cref{section:2-spectral-geometry-dai-yoshikawa} (see \Cref{setup:degeneration_with_reduced_fiber}).

Given a smooth real $(1, 1)$-form $\alpha$
on $X$, we introduce the following two integrability
conditions:
\begin{condition}
    \label[condition]{condition:integrability_general_fibers}
    We say $\alpha$ satisfies \emph{the integrability 
    condition on general fibers}, if 
    \(
        \int_{X_s} \alpha = 0
    \)
    for all $s \neq 0$.
\end{condition}

\begin{condition}
    \label[condition]{condition:integrability_singular_fibers}
    We say $\alpha$ satisfies \emph{the integrability 
    condition on the singular fiber}, if $\alpha$ satisfies \emph{the integrability 
    condition on general fibers (\Cref{condition:integrability_general_fibers})} and 
    \[
        \int_{C_i} \alpha = 0,
    \]
    where $C_i$ ($1 \leq i \leq N$) are irreducible components
    of the reduced singular fiber $X_0$. 
\end{condition}

Our goal in this section is to study the \emph{preferred potentials} of smooth $(1, 1)$-forms on $X$ satisfying \Cref{condition:integrability_general_fibers}. The notion of \emph{preferred potentials} originates from \cite{Filip-Tosatti21}*{Theorem 3.2.4}.
\begin{definition}[\bf{Preferred potentials}]
    \label{dfn:preferred_potential}
    Let $\alpha$ be a smooth real $(1, 1)$-form on $(X, \omega)$ satisfying \emph{the integrability 
    condition on general fibers (\Cref{condition:integrability_general_fibers})}.
    Then there is a unique $\vphi \in C^\infty(X\setminus X_0)$
    such that 
    \begin{align*}
        \int_{X_s} \vphi \omega_s  = 0,
        \quad 
        (\alpha + \ddc \vphi)\mid_{X_s} = 0.
    \end{align*}
    We say that $\vphi$ is the \emph{preferred potential} of the differential form $\alpha$.
\end{definition}

Since $(X_s, \omega_s)$ is compact and $\vphi\mid_{X_s}\eqqcolon \vphi_s$ 
is smooth on it, we can expand it using Fourier series 
in terms of eigenfunctions of $\Delta_{\omega_s}$:
\[
    \vphi_s = \sum_{i=1}^{N-1} c_i(s) \Phi_i(s, z)
    + \sum_{i\geq N} c_i(s) \Phi_i(s, z).
\]
The coefficients are determined by 
\begin{align*}
    c_i(s) &= \int_{X_s} \vphi_s \Phi_i(s, z) \omega_s\\
    &= \int_{X_s} \vphi_s \frac{-\ddc_z \Phi_i(s, z)}{\lambda_i(s)}\\
    &= \int_{X_s} \ddc_z(\vphi_s) \frac{-\Phi_i(s, z)}{\lambda_i(s)}\\
    &= \frac{1}{\lambda_i(s)}\int_{X_s} \Phi_i(s, z) \alpha.
\end{align*}
We use $b_i(s) \coloneqq \int_{X_s} \Phi_i(s, z) \alpha$ to denote the 
integral terms, and the properties of $\lambda_i(s)$ are known 
for $1 \leq i \leq N-1$ in \Cref{thm:eigenvalue_growth_rate}.

We then divide $\vphi$ into two parts, a low frequency part and a high 
frequency part, following the Fourier expansion:
\begin{align}
    \label{eqn:defn_of_freq_decomp}
    \vphi_{\rm low} = \sum_{i=1}^{N-1} c_i(s) \Phi_i(s, z),
    \quad 
    \vphi_{\rm high} = \sum_{i\geq N} c_i(s) \Phi_i(s, z).
\end{align}
These two parts are smooth on $X\setminus X_0$ by \Cref{lem:proj_onto_small_freq_is_smooth}.

Then we introduce the main theorems of this section.
\begin{theorem}
    \label{thm:boundedness_of_preferred_potential_general}
    We work under \Cref{setup:degeneration_with_reduced_fiber}.
    Suppose $\alpha$ is a smooth real $(1,1)$-form on $(X,\omega)$ with 
    \Cref{condition:integrability_general_fibers}. Let 
    $\vphi$ be the preferred potential of $\alpha$.
    \begin{enumerate}
        \item We have an $L^\infty$ bound for $\vphi_{\rm high}$,
        \[
            \norm{\vphi_{\rm high}}_{L^\infty(X_s)}
            \leq C \norm{\tr_{\omega_s} \alpha}_{L^\infty(X_s)}
        \]
        for all $s \neq 0$. The constant $C$ only depends
        on the geometry of the degeneration.
        \item We have control over the growth order of 
        $\norm{\vphi_{\rm low}}_{L^\infty(X_s)}$, 
        \[
            \norm{\vphi_{\rm low}}_{L^\infty(X_s)}
            \leq C \norm{\tr_{\omega_0} \alpha}_{L^1(X_0)}
            \log(\abs{s}^{-1}).
        \]
        for all $s \neq 0$.
        The constant $C$ only depends
        on the geometry of the degeneration.
    \end{enumerate}
\end{theorem}

When $\alpha$ satisfies the integrability condition
on the singular fiber (\Cref{condition:integrability_singular_fibers}),
the growth order of $\norm{\vphi_{\rm low}}_{L^\infty(X_s)}$ has a 
better control.
\begin{theorem}
    \label{thm:preferred_potential_integrability_singular_fiber}
    We use the same assumption and notations in \Cref{thm:boundedness_of_preferred_potential_general}. 
    We assume that the smooth real $(1,1)$-form $\alpha$ 
    satisfies \Cref{condition:integrability_singular_fibers}.
    Then the low frequency 
    part $\vphi_{\rm low}$ of the preferred potential satisfies
    \[
        \frac{\norm{\vphi_{\rm low}}_{L^\infty(X_s)}}{\log^{1/2}(\abs{s}^{-1})} \to 0, \quad
    \text{as } s \to 0.
    \]
\end{theorem}

\subsection{Estimate of the high frequency part}
We prove the first part of \Cref{thm:boundedness_of_preferred_potential_general} regarding the high frequency part of the \emph{preferred potential}. The proof 
is based on Cheng-Li's argument \cite{Cheng_Li1981}
(also seen in Siu's lecture note \cite{Siu_Lectures_On_KE}*{Appendix (A.2)}). 

First, we introduce the following high frequency truncated heat kernel and Green's function on each $(X_s, \omega_s)$:
\begin{align*}
    H^N_s(z_1,z_2,t) &= \sum_{i = N}^{\infty} e^{-\lambda_i(s) t} \Phi_i(s,z_1) \Phi_i(s,z_2) \quad (\text{truncated heat kernel})\\
    G^N_s(z_1,z_2) &=  \sum_{i = N}^{\infty} \frac{1}{\lambda_n(s)} \Phi_i(s,z_1) \Phi_i(s,z_2) \quad (\text{truncated Green's function})
\end{align*}

Then the high frequency part $\vphi_{\rm high}$ in 
\Cref{eqn:defn_of_freq_decomp} can be represented by 
\begin{equation}
    \label{eqn:high_part_by_green}
    \vphi_{\rm high}(z, s) = \int_{z_2 \in X_s} (\tr_{\omega_s}\alpha)(z_2) G^N_s(z, z_2) \omega_s(z_2)
\end{equation}

To obtain the desired bound on $\vphi_{\rm high}$ in \Cref{thm:boundedness_of_preferred_potential_general}, we estimate 
$G^N_s(z_1,z_2)$ using Cheng-Li's argument:
\bproposition
\label[proposition]{prop:bound_truncated_green_func}
There is a uniform constant $C > 0$ independent of $s \in \D^\circ$ 
such that 
\[
G^N_s(z_1, z_2) \geq -C, \quad z_1, z_2 \in X_s,
\]
for all $s \in \D^\circ$.
\eproposition
\begin{proof}
    Let $H_s(z_1, z_2, t)$ be the heat kernel of $(X_s, \omega_s)$,
    then we have 
    \begin{equation}
        \label{eqn:tmp_heat_knl_sum}
        \frac{1}{V(s)} + \sum_{i = 1}^{N-1} e^{-\lambda_i(s) t}\Phi_i(s,z_1) \Phi_i(s,z_2) +  H^N_s(z_1,z_2,t)
        = H_s(z_1, z_2, t) \geq 0,
    \end{equation}
    where $V(s) = \int_{X_s} \omega_s$ is the volume of $X_s$. Since 
    $\omega$ is closed, $V(s) \equiv V$ is constant.

    From the lower bound of higher eigenvalues in \Cref{thm:high_eigenvalues_lowe_bound}, we know that there
    exists $\lambda > 0$ such that
    $\lambda_k(s) \geq \lambda$ for all $k \geq N$. 
    Then 
    \begin{align*}
        \frac{\di }{\di t}H^N_s(z, z, t) = \sum_{i=N}^{\infty}-\lambda_i(s) e^{-\lambda_i(s) t} \abs{\Phi_i(s,z)}^2 \leq -\lambda H^N_s(z,z,t).
    \end{align*}
    By Grönwall's inequality, we have
    \[
    H^N_s(z,z,t) \leq H^N_s(z,z, 1) e^{-\lambda(t-1)}  \quad \text{for } t \geq 1
    \]
    By the Cauchy-Schwarz inequality, we get
    \begin{equation}
        \label{eqn:trunc_heat_exp}
        \begin{split}
        \abs{H^N_s(z_1,z_2,t)} &\leq \sqrt{H^N_s(z_1,z_1,t) \cdot H^N_s(z_2,z_2,t)}\\
        &\leq \sqrt{H^N_s(z_1,z_1,1) \cdot H^N_s(z_2,z_2,1)} e^{-\lambda(t - 1)} \quad \text{for } t \geq 1.
        \end{split}
    \end{equation}

    To estimate the truncated Green's function, we note 
    that 
    \[
    \int_{0}^{\infty} H^N_s(z_1,z_2,t) \di t = G^N_s(z_1,z_2) = \sum_{i = N}^{\infty} \frac{1}{\lambda_i(s)} \Phi_i(s,z_1) \Phi_i(s,z_2).
    \]
    Then 
    {\allowdisplaybreaks
    \begin{align*}
    G^N_s(z_1,z_2) &= \int_0^{1} H^N_s(z_1,z_2,t) \di t + \int_{1}^{\infty} H^N_s(z_1,z_2,t) \di t\\
    &\geq \int_0^1 H^N_s(z_1,z_2,t) \di t -  \sqrt{H^N_s(z_1,z_1,1) \cdot H^N_s(z_2,z_2,1)} \int_{1}^{\infty}e^{-\lambda(t - 1)} \di t\\
    &= \int_0^1 \bigl(H_s(z_1,z_2,t) - \frac{1}{V} - \sum_{i = 1}^{N-1} e^{-\lambda_i(s) t}\Phi_i(s,z_1) \Phi_i(s,z_2)\bigr) \di t - I_2\\
    &(\text{Using }\Cref{eqn:tmp_heat_knl_sum}, \text{and }I_2 \coloneqq \sqrt{H^N_s(z_1,z_1,1) \cdot H^N_s(z_2,z_2,1)} \int_{1}^{\infty}e^{-\lambda (t - 1)} \di t )\\
    &\geq \int_0^1 \bigl(- \frac{1}{V} - \sum_{i = 1}^{N-1} e^{-\lambda_i(s) t}\Phi_i(s,z_1) \Phi_i(s,z_2)\bigr) \di t - I_2\\
    &(\text{Using the fact that the heat kernel } H_s(z_1,z_2,t) \geq 0)
    \\
    &\geq -\frac{1}{V} - \sum_{i=1}^{N-1} \norm{\Phi_i(s,z)}^2_{L^\infty(X_s)} - 
    I_2\\
    &\geq - (\frac{1}{V} + I_1 + I_2)
\end{align*}}
Here 
\begin{align*}
    I_1 &= \sum_{i=1}^{N-1} \norm{\Phi_i(s,z)}^2_{L^\infty(X_s)} \geq 0\\
    I_2 &= \sqrt{H^N_s(z_1,z_1,1) \cdot H^N_s(z_2,z_2,0.4)} \int_{1}^{\infty}e^{-\lambda(t - 1)} \di t \geq 0.
\end{align*}

It remains to control $I_1, I_2$ by some uniform constants that are 
independent of $s$.

We begin with basic results which 
relate the Sobolev constants to the heat kernel and 
eigenfunction bounds.

\btheorem[Theorem 2.1, \cite{Yoshikawa1997}]\label{thm:heat-kernel-upper-bound} For a manifold \((M, g)\) of real dimension $2$, the following two inequalities are equivalent:
\begin{enumerate}
    \item For \(0 < t \leq 1\) and \((x,y) \in M \times M\),
    \[
        H(x,y,t) \leq C_1 t^{-2}
    \]

    \item For every \(f \in C^\infty(M)\),
    \[
        \norm{f}_{L^4} \leq C_2 (\norm{\di f}_{L^2}+ \norm{f}_{L^2})
    \]
\end{enumerate}
Here \(C_1\) and \(C_2\) depend continuously on each other.
\etheorem

By classical Moser iteration, we have
\btheorem [\cite{Petersen2016}*{Theorem 9.2.7}]
\label{thm:eigenfunctions_bounds_sobolev}
Let \((M,g)\) be a compact Riemannian manifold, with the following Sobolev inequality
\[
    \norm{u}_{2\nu} \leq S \norm{\nabla u}_2 + \norm{u}_2
\]
where \(\nu > 1\). Then for \(\Delta u = -\lambda u\), we have
\[
    \norm{u}_{\infty} \leq \exp(\frac{S\sqrt{\nu\lambda}}{\sqrt{\nu} - 1}) \norm{u}_2.
\]

Specifically, if $\nu = 2$ in the above Sobolev inequality, and $\norm{u}_2 = 
1$, then for $\Delta u = - \lambda u$, we have 
\[
  \norm{u}_{\infty} \leq \exp(\frac{S\sqrt{2\lambda}}{\sqrt{2} - 1})  
\] 
\etheorem

Finally, we introduce the uniform Sobolev inequality for a \Kahler family.
\btheorem [\cite{Tosatti2010}*{Lemma 3.2}]
\label{thm:uniform_sobolev_projective_family}
Let \(\pi:(X, \omega) \to \D\) be a \Kahler family, i.e., \(\pi\) is a proper, surjective holomorphic map from a \Kahler manifold \((X, \omega)\) to the unit disk \(\D\). Suppose
that $X_0 \coloneqq \pi^{-1}(0)$ is the unique singular fiber.

Assume \(\dim_\C X_s = 1\). Then for \(0 < \abs{s} < 0.5\), we have a uniform Sobolev constant $C$ for \((X_s,\omega_s)\), where \(\omega_s = \omega|_{X_s}\):
\[
    \norm{f}_{L^4(X_s)} \leq C(\norm{\di f}_{L^2} + \norm{f}_{L^2}) \quad \forall f \in C^\infty(X_s)
\]
\etheorem

So $I_1$ is bounded uniformly by \Cref{thm:uniform_sobolev_projective_family} and \Cref{thm:eigenfunctions_bounds_sobolev}. And $I_2$ is bounded uniformly
by \Cref{thm:uniform_sobolev_projective_family} and \Cref{thm:heat-kernel-upper-bound}. This completes the proof.
\end{proof}

As a corollary, we can show the first part of \Cref{thm:boundedness_of_preferred_potential_general}.
\bcorollary
Under the settings and notation of \Cref{thm:boundedness_of_preferred_potential_general},
we have 
\[
    \norm{\vphi_{\rm high}}_{L^\infty(X_s)}
    \leq C \norm{\tr_{\omega_s} \alpha}_{L^\infty(X_s)}
\]
for all $s \neq 0$. The constant $C$ only depends
on the geometry of the degeneration.
\ecorollary
\begin{proof}
    By \Cref{eqn:high_part_by_green}, we have 
    \begin{align*}
        \vphi_{\rm high}(z, s) &= \int_{z_2 \in X_s} (\tr_{\omega_s}\alpha)(z_2) G^N_s(z, z_2) \omega_s(z_2)\\
        &=\int_{z_2 \in X_s} (\tr_{\omega_s}\alpha)(z_2) (G^N_s(z, z_2) + C') \omega_s(z_2),
    \end{align*}
    where we take the constant $C'$ as in \Cref{prop:bound_truncated_green_func} such that 
    $G^N_s(z, z_2) + C' \geq 0$.

    So 
    \begin{align*}
        \vphi_{\rm high}(z, s) &\leq \int_{z_2 \in X_s} \abs{(\tr_{\omega_s}\alpha)(z_2)} (G^N_s(z, z_2) + C') \omega_s(z_2)\\
        &\leq \int_{z_2 \in X_s} 
        \norm{\tr_{\omega_s} \alpha}_{L^\infty(X_s)} (G^N_s(z, z_2) + C') \omega_s(z_2)\\
        &=C' V \norm{\tr_{\omega_s} \alpha}_{L^\infty(X_s)}.
    \end{align*}
    Similarly, we have $\vphi_{\rm high}(z, s) \geq -C' V\norm{\tr_{\omega_s} \alpha}_{L^\infty(X_s)}$, where $V$ is the volume of 
    $(X_s, \omega_s)$. 

    To summarize, we have $\norm{\vphi_{\rm high}}_{L^\infty(X_s)}
    \leq C \norm{\tr_{\omega_s} \alpha}_{L^\infty(X_s)}$.
\end{proof}

\vone
\subsection{Estimate of the low frequency (I)}
\label{section:low_freq_estimate_part_1}
We prove the second part of \Cref{thm:boundedness_of_preferred_potential_general} here.

First, we 
note that 
\[
    \vphi_{\rm low} = \sum_{i=1}^{N-1} c_i(s) \Phi_i(s, z)
\]
and the $\Phi_i(s, z)$ are uniformly bounded by \Cref{thm:uniform_sobolev_projective_family} and \Cref{thm:eigenfunctions_bounds_sobolev}. 
So it suffices to estimate 
$c_i(s)$. 
Recall that $c_i = \dfrac{b_i}{\lambda_i}$ and that the small eigenvalues $\lambda_i(s)$ are well-studied by Dai-Yoshikawa in \Cref{thm:eigenvalue_growth_rate}.
It remains to determine the asymptotic properties of 
\[
    b_i(s) = \int_{X_s} \Phi_i(s, z) \alpha
\]
around $0$ for all $1 \leq i \leq N$. 

Our tool is the decomposition
of $\Phi_i$ in \Cref{thm:approximate_eigenfunctions}. We recall
the results here: For each $1 \leq i \leq N-1$, we have 
\[
    \Phi_i(s, z) = \sum_{l=1}^{N}C_{li}(s) \Upsilon^{(l)}(s, z)
    + R_i (s, z).
\]
And 
\begin{itemize}
    \item There is a uniform constant $C$ such that $\abs{C_{li}(s)}
        < C$ for all $1 \leq l \leq N$ and all $s$.
    \item The model functions $\Upsilon^{(l)}(s, z)$ 
    are discussed in \Cref{prop:good_test_on_X_s}.
    \item The remainder terms $R_i$ are discussed in \Cref{thm:approximate_eigenfunctions}.
\end{itemize}
Then each $b_i$ is decomposed into 
\begin{equation}
    \label{eqn:b_i-decomposition}
    b_i(s) = \sum_{l=1}^N C_{li}(s) \int_{X_s} \Upsilon^{(l)}_s \alpha
    + \int_{X_s} R_i(s, z) \alpha.
\end{equation}

We have the following lemma
\blemma
\label[lemma]{lem:integral_of_model_function_with_forms}
Under the preceding notation. For any smooth
$(1, 1)$-form $\alpha$ on $X$,
\begin{equation*}
    \int_{X_s} \Upsilon^{(l)}_s \alpha = 
    \frac{1}{\sqrt{\mathrm{Area}(C_l)}}\bigl(\int_{C_l} \alpha +
    \norm{\alpha}_{\infty} O(\abs{s}^{\frac{1}{4\nu}})\bigr).
\end{equation*}
\elemma 
\bproof(of \Cref{lem:integral_of_model_function_with_forms}) From the construction of $\Upsilon^{(l)}_s$ in \Cref{prop:good_test_on_X_s}, it follows that 
\[
\int_{X_s} \Upsilon^{(l)}_s \alpha
= \int_{X_0} 
\frac{\chi_{\epsilon(s)}^{(l)} F_s^* \alpha}{\sqrt{\mathrm{Area}(C_l)}},
\]
where $\epsilon(s) = 2 \abs{s}^{\frac{1}{8 \nu}}$.

Since $\alpha$ is a smooth form on $X$, by 
\Cref{thm:dai-yoshikawa-package}(4) and \Cref{rmk:of_dai_yoshikawa_package}, we have 
\[
\norm{F_s^* \alpha - \alpha_0}_{L^\infty(X_0 \setminus \cup_{p \in \Sing X_0} B(p, \epsilon(s)) )} 
\leq K_2 \abs{s}^{\frac{1}{2}}
\]
for a uniform constant $K_2 > 0$. So 
\begin{equation}
    \label{eqn:tmp_1-lem3-11}
    \abs{\int_{X_s} \Upsilon^{(l)}_s \alpha - 
\int_{X_0} 
\frac{\chi_{\epsilon(s)}^{(l)} \alpha}{\sqrt{\mathrm{Area}(C_l)}}}
\leq K_2 \sqrt{\mathrm{Area}(C_l)} \abs{s}^{\frac{1}{2}}.
\end{equation}

Since $\chi_{\epsilon(s)}^{(l)}$ is compactly
supported on $C_l$, $\abs{\chi_{\epsilon(s)}^{(l)} - 1}$ is bounded by $2$ and it vanishes on $C_l$ except on a set of measure $O(\abs{s}^{\frac{1}{4\nu}})$ 
(by \Cref{lem:good_test_central})
, we have 
\begin{equation}
    \label{eqn:tmp_2-lem3-11}
    \abs{\int_{X_0} 
\frac{\chi_{\epsilon(s)}^{(l)} \alpha}{\sqrt{\mathrm{Area}(C_l)}} - 
\int_{C_l} 
\frac{\alpha}{\sqrt{\mathrm{Area}(C_l)}}
} = 
\norm{\alpha}_{\infty}O(\abs{s}^{\frac{1}{4\nu}}).
\end{equation}

Combining \Cref{eqn:tmp_1-lem3-11} and 
\Cref{eqn:tmp_2-lem3-11}, we obtain 
\[
\abs{\int_{X_s} \Upsilon^{(l)}_s \alpha - 
    \frac{1}{\sqrt{\mathrm{Area}(C_l)}}\int_{C_l} \alpha} =\norm{\alpha}_{\infty}O(\abs{s}^{\frac{1}{4\nu}}).
\]  
\eproof

On the other hand, from \Cref{thm:approximate_eigenfunctions}(2), we have 
\[
    \abs{\int_{X_s} R_i(s, z) \alpha} 
    = \norm{\tr_\omega \alpha}_{L^2(X_s)} 
    O(\log^{-\frac{1}{2}}(\abs{s}^{-1})).
\]

Putting all these asymptotic results into \Cref{eqn:b_i-decomposition}, we have
\begin{equation}
    \label{eqn:b_i-asymptotic}
    \begin{split}
    b_i(s) &= \sum_{l=1}^N \frac{C_{li}(s)}{
        \sqrt{\mathrm{Area}(C_l)}
    } \int_{C_l} \alpha
    + \sum_{l=1}^N \frac{C_{li}(s)}{
        \sqrt{\mathrm{Area}(C_l)}
    }  O(\abs{s}^{\frac{1}{4 \nu}})\\
    &+\norm{\tr_\omega \alpha}_{L^2(X_s)} 
    O(\log^{-\frac{1}{2}}(\abs{s}^{-1}))\\
    &= \sum_{l=1}^N \frac{C_{li}(s)}{
        \sqrt{\mathrm{Area}(C_l)}
    }\int_{C_l} \alpha  + O(\log^{-\frac{1}{2}}(\abs{s}^{-1}))\\
    &= O(1) \norm{\tr_{\omega_0}\alpha}_{L^1(X_0)} + O(\log^{-\frac{1}{2}}(\abs{s}^{-1})) \norm{\tr_\omega \alpha}_{L^2(X_s)}.
\end{split}
\end{equation}

We then prove the second part of \Cref{thm:boundedness_of_preferred_potential_general}.
\bproof(of \Cref{thm:boundedness_of_preferred_potential_general}.(2))
For $\alpha$ whose integral is zero on general fibers, following \Cref{eqn:b_i-asymptotic}, we have 
$b_i(s) = O(\log^{-\frac{1}{2}}(\abs{s}^{-1})) \norm{\tr_\omega \alpha}_{L^2(X_s)}.$

Since $\vphi_{\rm low} = \sum_{i=1}^{N-1} c_i(s) \Phi_i(s, z)$ and $c_i = b_i/\lambda_i$, we have 
\[
c_i(s) = O(\log(\abs{s}^{-1}))\norm{\tr_{\omega_0}\alpha}_{L^1(X_0)} + O(\log^{\frac{1}{2}}(\abs{s}^{-1}))\norm{\tr_\omega \alpha}_{L^2(X_s)}.
\]

Since $\Phi_i(s, z)$ is uniformly bounded for $1 
\leq i \leq N - 1$, we have 
\[
    \norm{\vphi_{\rm low}}_{L^\infty(X_s)}
    \leq C \norm{\tr_{\omega_0}\alpha}_{L^1(X_0)}
    \log(\abs{s}^{-1}).
\]
\eproof

\subsection{Estimate of the low frequency (II)}
If $\alpha$ additionally satisfies \Cref{condition:integrability_singular_fibers} on the singular fiber (i.e., $\alpha$ has zero integral on each irreducible component of the singular fiber), we obtain a better estimate on its preferred potential (\Cref{thm:preferred_potential_integrability_singular_fiber}).
We prove the theorem here.

Following the same strategy as in \Cref{section:low_freq_estimate_part_1}, it suffices 
to estimate each term in the expansion \Cref{eqn:b_i-decomposition} of $b_i(s)$:
\[
b_i(s) = \sum_{l=1}^N C_{li}(s) \int_{X_s} \Upsilon^{(l)}_s \alpha
    + \int_{X_s} R_i(s, z) \alpha.
\]

The first part is good enough when $\alpha$ has 
zero integral on each irreducible component of the singular fiber as proved
in \Cref{lem:integral_of_model_function_with_forms}, and we have 
\[
\sum_{l=1}^N C_{li}(s) \int_{X_s} \Upsilon^{(l)}_s \alpha = O(\abs{s}^{\frac{1}{4 \nu}}).
\]

It remains to estimate $\int_{X_s} R_i(s, z) \alpha$ 
under conditions \Cref{condition:integrability_singular_fibers}, which
is the following proposition.

\bproposition
For $s \neq 0$, we define $\tilde{R}_i \coloneqq
\log^{\frac{1}{2}}(\abs{s}^{-1}) R_i$. If $\alpha$
is a smooth $(1, 1)$-form on $X$ with \Cref{condition:integrability_singular_fibers}, we then have  
\[
    \log^{\frac{1}{2}}(\abs{s}^{-1})\abs{\int_{X_s} R_i(s, z) \alpha} 
    = \abs{\int_{X_s} \tilde{R}_i(s, z) \alpha} \to 0
    \quad \text{as } s \to 0.
\]
\eproposition
\begin{proof}
    We prove it by contradiction. Suppose we have a sequence 
    $s_n \to 0$ such that 
    $\abs{\int_{X_{s_n}} \tilde{R}_i(s_n, z) \alpha} \geq \epsilon_0$
    for some $\epsilon_{0} > 0$.

    First, we check that the family of functions
    $\tilde{R}_{i, s}(z) = \tilde{R}_i(s, z) \in C^\infty(X_s)$ satisfies the
    conditions in \Cref{prop:central_convergence_compactness}. We denote
    $\ddc_z \tilde{R}_{i, s}(z)$ by $\beta_s$.
    \begin{itemize}
        \item For the first condition in \Cref{prop:central_convergence_compactness}, we use the 
        decomposition formula for eigenfunctions $\Phi_i$ in 
        \Cref{eqn:eigenfunctions_correlation}:
        \[\Phi_i(s, z) = \sum_{l=1}^{N}C_{l(i+1)}(s) \Upsilon^{(l)}(s, z)
        + R_i (s, z).\]
        Around any $x \in X_{0, \rm reg}$, we can take an open neighborhood
        $U$ of $x$ such that $\Upsilon^{(l)}(s, z)$
        is constant on $U$ by \Cref{cor:local_constant_of_upsilon}.
        Then on $U \cap X_s$, we have 
        \[
        \ddc_z \tilde{R}_{i, s}(z) = \beta_s = \ddc_z \Phi_i(s, z)
        = -\lambda_i(s) \Phi_i(s, z) \omega_s.
        \]
        So $\norm{\tr_{\omega_s}\beta_s}_{L^\infty(U \cap X_s)}
        = \norm{\lambda_i(s) \Phi_i(s, z)}_{L^\infty(U \cap X_s)}$
        is uniformly bounded in $s$.
        \item For the second condition in \Cref{prop:central_convergence_compactness}, we note that 
        $$\norm{\tr_{\omega_s}\beta_s}_{C^0(U \cap X_s)} = 
        \lambda_i(s) \norm{\Phi_i}_{C^0(U \cap X_s)} \to 0.$$
        So $\tr_{\omega_s}\beta_s \to 0 = f_0$ in $C^{0}_{\rm loc}$
        on $X_{0, \rm reg}$.
        \item For the third condition in \Cref{prop:central_convergence_compactness}, it is proved in 
        \Cref{thm:approximate_eigenfunctions}(3) that there is 
        a uniform constant $C > 0$ such that 
        \[
            \norm{\di R_j}^2_{L^2(X_s)} \leq \frac{C^2}{\log(\abs{s}^{-1})},
            \quad
            \norm{R_j}^2_{L^2(X_s)} \leq C^2 \frac{1}{\log(\abs{s}^{-1})}.
        \]
        for all $1 \leq j \leq N$ and all $s$. So 
        \[
        \norm{\tilde{R}_{i, s}(z)}_{L^2(X_s, \omega_s)}, 
        \norm{\di \tilde{R}_{i, s}(z)}_{L^2(X_s, \omega_s)} \leq C^2.
        \]
    \end{itemize}

    We then apply \Cref{prop:central_convergence_compactness} 
    to the family of functions $\tilde{R}_{i, s}(z)$ and sequence
    $s_n \to 0$, we then have a subsequence $s_{k_n}$ such that 
    \begin{itemize}
        \item $\tilde{R}_{i,s_{k_n}}(z)$ converges to $\tilde{R}_{0}(z)$
        in $C^{1, a}_{\rm loc}$ on $X_{0, \rm reg}$ for any 
        $0 < a < 1$,
        \item $\tilde{R}_{0}(z) \in W^{2, p}_{\rm loc}$ for all $1 < p < \infty$ 
        and it solves $-\Delta_{\omega_0} \tilde{R}_{0}(z) = f_0 = 0$ a.e. on $X_{0, \rm reg}$,
        \item $\tilde{R}_{0}(z) \in W^{1, 2}
        (X_{0, \rm reg})$,
        \item $\int_{X_{s_{k_n}}} \tilde{R}_{i,s_{k_n}}(z) \alpha 
        \to \int_{X_{0}} \tilde{R}_{0}(z) \alpha.$
    \end{itemize}

    Since $\tilde{R}_{0}(z) \in W^{1, 2}
    (X_0)$ and it is harmonic on $X_{0, \rm reg}$, it is locally constant 
    and is constant on each irreducible component
    of $X_0$. So we have $\int_{X_{0}} \tilde{R}_{0}(z) \alpha = 0$. 
    Thus the initial assumption $\abs{\int_{X_{s_n}} \tilde{R}_i(s_n, z) \alpha} \geq \epsilon_0$
    for all $n$ and an $\epsilon_{0} > 0$ contradicts that  
    $$\int_{X_{s_{k_n}}} \tilde{R}_{i,s_{k_n}}(z) \alpha 
        \to \int_{X_{0}} \tilde{R}_{0}(z) \alpha = 0.$$ 
    
    This completes the proof.
\end{proof}

As a direct consequence, we can prove \Cref{thm:preferred_potential_integrability_singular_fiber}
\begin{proof}(of \Cref{thm:preferred_potential_integrability_singular_fiber})
    Following previous discussion, we have 
    \[
    b_i(s) = O(\abs{s}^{\frac{1}{4\nu}}) + 
    o(\log^{-\frac{1}{2}}(\abs{s}^{-1})).
    \]
    So $c_i(s) = \frac{b_i(s)}{\lambda_i(s)} = o(\log^{\frac{1}{2}}(\abs{s}^{-1}))$ for $1 \leq i \leq N - 1$. It implies that 
    \[
        \frac{\norm{\vphi_{\rm low}}_{L^\infty(X_s)}}{\log^{1/2}(\abs{s}^{-1})} \to 0, \quad
    \text{as } s \to 0.
    \]
\end{proof}

\section{Asymptotics of Archimedean height pairing}
\label{section:4-Archimedean-pairing}
As an application of previous estimates,
we introduce \emph{the Archimedean height pairing for 
smooth real $(1,1)$-forms on $X$}. We continue to use the setup 
from \Cref{setup:degeneration_with_reduced_fiber}, except in 
\Cref{cor:continuity_of_pairing_general_fiber}, where we allow 
for general singular fibers of the degeneration.

\begin{definition}[Archimedean height pairing]
    \label{dfn:pairing_of_forms}
    Let $\alpha, \beta$ be two smooth real $(1, 1)$-forms 
    on $X$ with $\int_{X_s} \alpha = \int_{X_s} \beta = 0$
    for $s \neq 0$ (\Cref{condition:integrability_general_fibers}). Let $\vphi$ be the preferred potential 
    of $\alpha$. We define \emph{the Archimedean height pairing}
    $\la \alpha, \beta \ra$ to be a function on $\D^\circ$
    \[
    \la \alpha, \beta \ra(s) \coloneqq \int_{X_s} \vphi \beta,
    \quad s \in \D^\circ.
    \]
    It is smooth on $\D^\circ$ since $\pi\colon X^\circ(\coloneqq 
    X \setminus X_0) 
    \to \D^\circ$ is a submersion and $\vphi$ is smooth. 
\end{definition}
\begin{remark}
    We can use any potential for $\alpha$ on $X_s$ to define the 
    Archimedean height pairing. Actually, for any potential $-\vphi'$ 
    of $\alpha$ on $X_s$(i.e., $\ddc \vphi' + \alpha = 0$ on 
    $X_s$), the difference $\vphi' - \vphi$ is constant on $X_s$, where $\vphi$
    is the preferred potential of $\alpha$. So $\int_{X_s}(\vphi' - \vphi)
    \beta = \mathrm{constant} \cdot \int_{X_s} \beta
    = 0$ and $\int_{X_s} \vphi' \beta =\int_{X_s} \vphi \beta =  
    \la \alpha, \beta \ra(s)$.
\end{remark}
We now state some basic properties of the Archimedean height pairing defined above.
\begin{proposition}
    \label[proposition]{prop:properties_of_pairing}
    Let $\alpha, \beta$ be two forms 
    as in \Cref{dfn:pairing_of_forms}. We have
    \begin{enumerate}
        \item The Archimedean height pairing
        is symmetric and bilinear.
        \item The Archimedean height pairing is 
        positive definite, i.e.,
        $\la \alpha, \alpha \ra$ is a non-negative
        function on $\D^\circ$, and 
        $\la \alpha, \alpha \ra 
        = 0$ if and only if $\alpha\mid_{X_s} = 0$
        for $s \neq 0$.
    \end{enumerate}
\end{proposition}
\begin{proof}\mbox{}
    \begin{enumerate}
        \item Since the preferred potentials depend
        linearly on the differential forms, we have 
        that the 
        Archimedean height pairing is bilinear. It is symmetric
        because \(
        \la \alpha, \beta \ra(s) = \int_{X_s}\vphi (-\ddc \vphi')
        =  \int_{X_s}\vphi' (-\ddc \vphi) = \la \beta, \alpha \ra(s),
        \)
        where $\vphi, \vphi'$ are preferred potentials for $\alpha,
        \beta$ respectively.
        \item We have 
        $\la \alpha, \alpha \ra(s)  = \int_{X_s}\vphi (-\ddc \vphi)
        = \norm{\di \vphi}^2_{L^2(X_s)} \geq 0$.
        If $\la \alpha, \alpha \ra = 0$, then 
        the preferred potential $\vphi$ of $\alpha$
        is fiberwise constant, so $\alpha|_{X_s}
        = 0$ for $s \neq 0$.
    \end{enumerate}
\end{proof}

\begin{theorem}[Continuity of pairing]
    \label{thm:continuity_of_pairing_zero_singular_integral}
    Let $\alpha, \beta$ be two smooth real $(1, 1)$-forms on $X$ with \Cref{condition:integrability_general_fibers} satisfying 
    \[
        \int_{C_i} \alpha = \int_{C_i} \beta = 0,
    \]
    where the $C_i$ are the irreducible components
    of the reduced singular fiber $X_0$ (\Cref{condition:integrability_singular_fibers}).

    Then the pairing function $\la \alpha, \beta \ra \in C^0(\D^\circ)$
    (in \Cref{dfn:pairing_of_forms})
    can be continuously extended to $\D$.
\end{theorem}
\begin{proof}(of \Cref{thm:continuity_of_pairing_zero_singular_integral})
    Let $\vphi, \rho \in C^\infty(X\setminus X_0)$ be preferred potentials of $\alpha$ and $\beta$ respectively.

    By \Cref{lem:proj_onto_small_freq_is_smooth} and shrinking 
    the base, we decompose $\vphi$ into 
    its low frequency part $\vphi_{\rm low} \in C^\infty(X\setminus X_0)$
    and high frequency part 
    $\vphi_{\rm high} \in C^\infty(X\setminus X_0)$. Then 
    \begin{align*}
        \la \alpha, \beta\ra(s)
        &=-\int_{X_s} (\vphi_{\rm low} + \vphi_{\rm high}) \ddc \rho \\
        &=-\int_{X_s} \rho \ddc \vphi_{\rm low} - 
        \int_{X_s}\vphi_{\rm high} \beta
    \end{align*}

    Using the asymptotic results from \Cref{thm:preferred_potential_integrability_singular_fiber}, we have 
    $\rho = o(\log^{\frac{1}{2}}(\abs{s}^{-1}))$ and 
    $\vphi_{\rm low}= o(\log^{\frac{1}{2}}(\abs{s}^{-1}))$.
    Since $\vphi_{\rm low}$ is in the space spanned by 
    eigenfunctions with low frequency, we have 
    \[
    \norm{\tr_{\omega_s} \ddc \vphi_{\rm low}}_{L^\infty(X_s)}
    = \norm{\Delta_{\omega_s} \vphi_{\rm low}}_{L^\infty(X_s)}
    = o(\log^{-\frac{1}{2}}(\abs{s}^{-1})).
    \]
    So $\int_{X_s} \rho \ddc \vphi_{\rm low} = o(1)$ and it is 
    continuous at $0$.

    To show that $\int_{X_s}\vphi_{\rm high} \beta$ can be continuously extended to $\D$, we use \Cref{cor:fiber_int_continuity_central_convergence}. It suffices 
    to check that $\vphi_{\rm high, s}\in C^\infty(X_s)$
    satisfies the conditions in \Cref{prop:central_convergence_compactness}. Note that 
    \[
    \ddc_z \vphi_{\rm high, s}(z)
    = -\alpha - \ddc_z \vphi_{\rm low, s}(z)
    \]
    \begin{itemize}
        \item For the first condition, we have 
        \begin{align*}
            \norm{\tr_{\omega_s}\ddc_z \vphi_{\rm high, s}(z)}_{L^\infty
            (X_s)}
            &\leq 
            \norm{\tr_{\omega_s}\alpha}_{L^\infty(X_s)}
            + \norm{\tr_{\omega_s}\ddc_z \vphi_{\rm low, s}}_{L^\infty(
                X_s
            )}\\
            &\leq \norm{\tr_{\omega_s}\alpha}_{L^\infty(X_s)}
            + C\log^{-1}(\abs{s}^{-1}) \norm{\vphi_{\rm low, s}}_{L^\infty
            (X_s)}\\
            &\leq \norm{\tr_{\omega_s}\alpha}_{L^\infty(X_s)} + 
            o(\log^{-\frac{1}{2}}(\abs{s}^{-1}))
        \end{align*}
        So $\norm{\tr_{\omega_s}\ddc_z \vphi_{\rm high, s}(z)}_{L^\infty
        (X_s)}$ is uniformly bounded in $s$.
        \item
        Since $\tr_{\omega_s}\alpha$ is continuous on $X \setminus
        \Sing X_0$ and 
        \[
        \norm{\tr_{\omega_s}\ddc_z \vphi_{\rm high, s}(z) - 
        \tr_{\omega_s}\alpha}_{L^\infty(X_s)}
        = \norm{\Delta_{\omega_s}\vphi_{\rm{low}, s}}_{L^\infty(X_s)} 
        =
        o(\log^{-\frac{1}{2}}(\abs{s}^{-1})).
        \]
        We have $\tr_{\omega_s}\ddc_z \vphi_{\rm high, s} \to 
        \tr_{\omega_0}\alpha$ in $C^0_{\rm loc}$ on 
        $X_{0, \rm reg}$, which is the second
        condition.
        \item Since $\norm{\vphi_{\rm high}}_{L^\infty(X_s)}$ (by \Cref{thm:boundedness_of_preferred_potential_general}) and $\norm{\tr_{\omega_s}\ddc_z \vphi_{\rm high, s}(z)}_{L^\infty (X_s)}$ are uniformly bounded, $\norm{\vphi_{\rm high}}_{L^2(X_s)}$ and $\norm{\di \vphi_{\rm high}}_{L^2(X_s)}$ are also uniformly bounded. So the third condition 
        is verified.
    \end{itemize}

    So we can apply \Cref{cor:fiber_int_continuity_central_convergence}
    to $\vphi_{\rm high, s}$ and $\int_{X_s}\vphi_{\rm high} \beta$
    is then continuous at $0$.

    In summary, $\la \alpha, \beta\ra(s)$ can be continuously extended to $\D$.
\end{proof}
\vone

For general smooth classes $\alpha$ and $\beta$ satisfying \Cref{condition:integrability_general_fibers}, we can extend \Cref{thm:continuity_of_pairing_zero_singular_integral}.

\begin{theorem}[Asymptotics of pairing]
    \label{thm:continuity_of_pairing_full}
    Let $\alpha$ and $\beta$ be two smooth real $(1, 1)$-forms whose integrals on general fibers are $0$. 
    Then there is a constant $c_{\alpha, 
    \beta}$ such that 
    \[
    \la \alpha, \beta \ra - c_{\alpha, \beta} \log\abs{s}^2
    \]
    extends continuously to $\D$. 
    
    The constant $c_{\alpha, \beta}$ is determined in the following
    way: Let $M$ be the intersection matrix of irreducible components
    of the singular fiber $X_0$, given by $M_{ij} = C_i \cdot C_j$. Let 
    $M^+$ be the 
    Moore-Penrose pseudoinverse of $M$. Then $c_{\alpha, \beta} = 
    \mathbf{v}_\alpha^T (M^+)^T  \mathbf{v}_\beta$, where 
    \[
    \mathbf{v}_\alpha = \left( \int_{C_1} \alpha, \dots, \int_{C_N} \alpha \right)^T, \quad 
    \mathbf{v}_\beta = \left( \int_{C_1} \beta, \dots, \int_{C_N} \beta \right)^T.
    \]

    In particular, if either $\alpha$ or $\beta$ satisfies the \emph{integrability condition on the singular fiber} (\Cref{condition:integrability_singular_fibers}), then $c_{\alpha, \beta }= 0$ because either $\mathbf{v}_\alpha$ or $\mathbf{v}_\beta$ is zero.
\end{theorem}
\begin{proof}
    Since the singular fiber is reduced, we have 
    $X_0=\sum_{i=1}^N C_i$,
    where $C_i$ are the irreducible components.

    We consider the $N \times N$ intersection matrix
    $M$ with entries $M_{ij} = C_i \cdot C_j$. By
    Zariski's Lemma (see \cite{BHPVdV2004}*{III.8.2, p.111}), we have 
    \begin{itemize}
        \item The matrix $M$ is negative semi-definite.
        \item The kernel of $M$ is exactly one-dimensional and is spanned by 
        $\mathbf{1} = (1, 1, \cdots, 1)$.
    \end{itemize}

    Let $V_0 = \left\{ \mathbf{x} \in \mathbb{R}^N \colon \sum_{i=1}^N x_i = 0 \right\}$ be the degree-zero subspace. Then $M|_{V_0}$
    is strictly negative definite and is invertible.
    To compute the inverse of $M|_{V_0}$, we introduce $M^+$ the 
    Moore-Penrose pseudoinverse of $M$. Then $M^+$ acts as the
    inverse of $M$ on $V_0$ and annihilates $\mathrm{span}(\mathbf{1})$.

    For each $C_i$, the irreducible component of 
    $X_0$, we consider the line bundle 
    $L_i = \mathcal{O}(C_i)$ equipped with a 
    Hermitian metric $h_i$. Then we have a closed smooth $(1, 1)$-form 
    $\gamma_i = -\ddc \log \norm{e}^2_{h_i}$ 
    representing the class $[C_i]$,
    where $e$ is a local frame of the holomorphic 
    line bundle $L_i$. Let $\sigma_i$ be the canonical section of $L_i$ which vanishes exactly on $C_i$. Then the preferred potential 
    for $\gamma_i$ on each smooth fiber $X_s$ is $\log \norm{\sigma_i}^2_{h_i}$ up to a constant.

    We define the following map, which maps onto
    differential
    forms representing cohomology classes spanned by $\{[C_i]\}$:
    \begin{align*}
    \mathrm{Form} \colon 
    \R^N
    &\to \{\text{smooth closed } (1,1)\text{-forms with zero integral on the singular fiber}\}\\
    \mathbf{x} = (x_1, \cdots, x_N) & \mapsto 
    \sum_{i=1}^N x_i \gamma_i \eqqcolon \mathrm{Form}(\mathbf{x})
    \end{align*}

    Furthermore, the preferred potential of $\mathrm{Form}(\mathbf{x})$ on $X_s$ will be $\sum_{i=1}^N x_i \log \norm{\sigma_i}^2_{h_i}$ up to a constant depending on $s \neq 0$.

    For smooth $(1, 1)$-forms $\alpha, \beta$ on $X$ such that $\int_{X_s} \alpha = \int_{X_s} \beta = 0$ for $s \neq 0$, the vectors of component integrals
    \[
    \mathbf{v}_\alpha = \left( \int_{C_1} \alpha, \dots, \int_{C_N} \alpha \right)^T, \quad 
    \mathbf{v}_\beta = \left( \int_{C_1} \beta, \dots, \int_{C_N} \beta \right)^T.
    \]
    are in $V_0$. We can then decompose $\alpha, \beta$
    into 
    \begin{align*}
        \alpha &= \mathrm{Form}(M^+ \mathbf{v}_\alpha) + \underbrace{(\alpha - \mathrm{Form}(M^+ \mathbf{v}_\alpha))}_{\eqqcolon \alpha_{0}} \\
        \beta &= \mathrm{Form}(M^+ \mathbf{v}_\beta) + \underbrace{(\beta - \mathrm{Form}(M^+ \mathbf{v}_\beta))}_{\eqqcolon \beta_{0}}
    \end{align*}
    Note that for each $C_i$,
    \begin{align*}
        \int_{C_i}\alpha_{0} &= 
        \text{i-th coordinate of }(\mathbf{v}_\alpha - M M^+ \mathbf{v}_\alpha) \\
        &= 0 
        \quad (\mathbf{v}_\alpha \in V_0, 
        \text{ so }M M^+ \mathbf{v}_\alpha = \mathbf{v}_\alpha).
    \end{align*}
    Similarly, we have $ \int_{C_i}\beta_{0} = 0$ for each 
    $C_i$.

    Since the pairing $\la \alpha, \beta\ra$
    is bilinear and symmetric, we have 
    \begin{align*}
        \la \alpha, \beta \ra
        &= \la \alpha_{0},\beta_{0}\ra +\la \alpha_{0},  \mathrm{Form}(M^+ \mathbf{v}_\beta)\ra+\la \mathrm{Form}(M^+ \mathbf{v}_\alpha), 
        \beta_{0}\ra\\
        &+\la \mathrm{Form}(M^+ \mathbf{v}_\alpha),
        \mathrm{Form}(M^+ \mathbf{v}_\beta)\ra
    \end{align*}

    By the previous theorem (\Cref{thm:continuity_of_pairing_zero_singular_integral}),
    $\la \alpha_{0},  \beta_{0}\ra$ can be extended 
    continuously to
    $S$. 
    
    Since the preferred potential for $\mathrm{Form}(M^+ \mathbf{v}_\alpha)$ is 
    $\sum_{i=1}^N (M^+ \mathbf{v}_\alpha)_i 
    \log \norm{\sigma_i}^2_{h_i}$ up to a constant on $X_s$, we have 
    \begin{align*}
        \la \mathrm{Form}(M^+ \mathbf{v}_\alpha), 
        \beta_{0}\ra =\sum_{i=1}^N\int_{X_s}  (M^+ \mathbf{v}_\alpha)_i 
        \log \norm{\sigma_i}^2_{h_i} \beta_{0}
    \end{align*}

    Using the asymptotics of log-norm integral (\Cref{lem:log_norm_section_fib_int_general_case}), each $\int_{X_s}\log \norm{\sigma_i}^2_{h_i} \beta_{0}$ extends continuously 
    on $S$. Thus, $\la \mathrm{Form}(M^+ \mathbf{v}_\alpha), 
    \beta_{0}\ra$ extends Hölder continuously 
    to $S$. 
    
    Similarly, $\la \alpha_{0},  \mathrm{Form}(M^+ \mathbf{v}_\beta)\ra$
    extends Hölder continuously 
    to $S$.

    For the last term, we have 
    \begin{align*}
        \la \mathrm{Form}(M^+ \mathbf{v}_\alpha),
        \mathrm{Form}(M^+ \mathbf{v}_\beta)\ra 
        &= 
        \sum_{i=1}^N \sum_{j=1}^N
        (M^+ \mathbf{v}_\alpha)_i 
        (M^+ \mathbf{v}_\beta)_j
        \int_{X_s}  
        \log \norm{\sigma_i}^2_{h_i} 
        \gamma_j.
    \end{align*}
    By \Cref{lem:log_norm_section_fib_int_general_case},  $\int_{X_s}  
    \log \norm{\sigma_i}^2_{h_i} 
    \gamma_j - \int_{C_i}\gamma_j \log\abs{s}^2$
    can be extended Hölder continuously through the 
    origin, so 
    \[
    \la \mathrm{Form}(M^+ \mathbf{v}_\alpha),
        \mathrm{Form}(M^+ \mathbf{v}_\beta)\ra  
        - \underbrace{\left(\sum_{1 \leq i,j \leq N}
        (M^+ \mathbf{v}_\alpha)_i 
        (M^+ \mathbf{v}_\beta)_j \int_{C_i}\gamma_j\right)}_{\eqqcolon c_{\alpha, \beta}} \log\abs{s}^2
    \] 
    extends Hölder continuously on $S$. Here $c_{\alpha, \beta}$ can be simplified as  
    \begin{align*}
        c_{\alpha, \beta} = 
        \sum_{1 \leq i,j \leq N}
        (M^+ \mathbf{v}_\alpha)_i 
        (M^+ \mathbf{v}_\beta)_j M_{ij}
        &= (M^+ \mathbf{v}_\alpha)^T M M^+ \mathbf{v}_\beta\\
        &=(M^+ \mathbf{v}_\alpha)^T {v}_\beta
    \end{align*}

    In summary, $\la \alpha, \beta \ra -c_{\alpha, \beta} \log\abs{s}^2$
    extends continuously to $\D$.
\end{proof}

We note that the height pairing $\la \alpha, \beta \ra$ 
is well-behaved under semi-stable reduction. So we can extend
the asymptotic result in \Cref{thm:continuity_of_pairing_full} to the
degeneration of algebraic curves with 
non-reduced singular fibers.
\begin{corollary}
    \label[corollary]{cor:continuity_of_pairing_general_fiber}
    Let $\pi\colon X \to \D$ be a degeneration of algebraic curves.
    Assume $X$ is smooth, and allow the singular fiber $X_0$ to be non-reduced and non-irreducible.
    Let
    $p \colon Y \to \D$ be a semi-stable reduction of $(X, X_0)$ with the following commutative diagram
    \[
    \begin{tikzcd}
        & Y  \arrow[r, "\mu"] \arrow[d, "p"] & X \arrow[d, "\pi"] \\
        & \D_t \arrow[r, "f_d"] & \D_s
    \end{tikzcd}
    \]
    Then $f_d(t) = t^d$ and $Y_0$ is a reduced 
    divisor with simple normal crossings. Furthermore, 
    $Y$ is a \Kahler manifold.

    Let $\alpha, \beta$ be two smooth real $(1, 1)$-forms on $X$
    with \Cref{condition:integrability_general_fibers}. Then 
    $\mu^* \alpha$ and $\mu^* \beta$ are two smooth $(1, 1)$-forms on $Y$ with \Cref{condition:integrability_general_fibers}. So we can define 
    $\la \alpha, \beta \ra \in C^\infty(\D_s^\circ)$ and 
    $\la \mu^* \alpha, \mu^* \beta \ra \in 
    C^\infty(\D_t^\circ)$. These two pairings are related 
    by 
    \begin{equation}
        \label{eqn:pairing_pullback}
        \la \alpha, \beta \ra \circ f_d = \la \mu^* \alpha, \mu^* \beta \ra.
    \end{equation}

    In particular, there is a constant $c_{\alpha, 
    \beta} \in \R$ such that 
    \(
    \la \alpha, \beta \ra - c_{\alpha, \beta} \log\abs{s}^2
    \)
    extends continuously to $\D$. So $\la \alpha, \beta \ra \in L^\infty(\D^\circ)$ if and only if 
    $\la \alpha, \beta \ra \in C^0(\D)$.
\end{corollary}
\begin{proof}
    It suffices to prove \Cref{eqn:pairing_pullback}. 
    
    For 
    $0 \neq t_0 \in \D_t$, we have an isomorphism $\mu_{t_0}
    \colon Y_{t_0} \to X_{f_d(t_0)}$. Let $-\vphi$ be a potential
    for $\alpha$ on $X_{f_d(t_0)}$, i.e., 
    $\ddc \vphi + \alpha = 0$ on $X_{f_d(t_0)}$. Then 
    $-\mu^* \vphi$ is a potential
    for $\mu^*\alpha$ on $Y_{t_0}$. So 
    \[
    \la \mu^* \alpha, \mu^* \beta \ra(t_0)
    = \int_{Y_{t_0}} (\mu^* \vphi) \mu^* \beta
    = \int_{X_{f_d(t_0)}} \vphi \beta = 
    \la \alpha, \beta \ra(f_d(t_0)).
    \]
    Thus $\la \alpha, \beta \ra \circ f_d = \la \mu^* \alpha, \mu^* \beta \ra$ on $\D_t^\circ$.

    By \Cref{thm:continuity_of_pairing_full}, $\la \mu^* \alpha, \mu^* \beta \ra  - c_{\mu^* \alpha, \mu^* \beta}
    \log\abs{t}^2$ is continuous on $\D_t$. So 
    we can find $c_{\alpha, \beta} \in \R$ depending on 
    $c_{\mu^* \alpha, \mu^* \beta}$ and $d$ such that 
    $\la \alpha, \beta \ra - c_{\alpha, \beta} \log\abs{s}^2$
    is continuous on $\D_s$.
\end{proof}

At the end of this section, we propose a question on the
Hölder continuity of the pairing.
\begin{question}
\label[question]{question:holder-continuity}
We consider a degeneration of algebraic curves with reduced singular
fibers as in \Cref{setup:degeneration_with_reduced_fiber}. Let $\alpha, \beta$ be two smooth real $(1, 1)$-forms on $X$ satisfying 
\(
    \int_{C_i} \alpha = \int_{C_i} \beta = 0,
\)
where $C_i$ are irreducible components
of the reduced singular fiber $X_0$. Is the pairing
$\la \alpha, \beta \ra$ a Hölder continuous function?

If this question is answered positively, we can strengthen the results in \Cref{thm:continuity_of_pairing_full}, 
\Cref{cor:continuity_of_pairing_general_fiber}, \Cref{prop:continuity_of_current_val_pairing} and \Cref{cor:dynamical_consequences} by replacing the word ``\emph{continuously}''
with ``\emph{Hölder continuously}''.
\end{question} 

\section{Application to K3 surfaces with parabolic automorphisms}
\label{section:5-k3-dynamics}

As an application of our previous results, we study the parabolic
dynamics of elliptic K3 surfaces, following 
\cite{Filip-Tosatti21}*{Section 3}. 

Let \(\pi \colon X \to \P^1 \eqqcolon B\) be a projective
elliptic K3 surface.
A \emph{parabolic automorphism} $T$ of $X$ is an automorphism of $X$
which preserves the fibration $\pi$ and is of infinite order. 
This yields the following commutative diagram:
\[
\begin{tikzcd}
    X \arrow[r, "T"]   \arrow[d, "\pi"']  & X
    \arrow[ld, "\pi"] \\
    B \coloneqq\mathbb{P}^1   &             
\end{tikzcd}
\]

In some literature (e.g. \cite{filip2022k3}*{6.2.2}), \emph{a parabolic automorphism} is also 
called a \emph{twist automorphism}.

Compared to \cite{Filip-Tosatti21}, we allow possibly
non-irreducible and non-reduced singular fibers of $\pi\colon X \to B$ 
in the setup of this section. 

We use $B^\circ$ to denote the smooth locus of 
the fibration $\pi$,
and we use $s$ to denote the local coordinate on $B$. We 
use $X_s \coloneqq \pi^{-1}(s)$.

\subsection{Preliminaries}
Before delving deeper, we introduce basic notions and lemmas in 
this part.

Let $[E] \in \NS(X)$ be the class of a general fiber of $\pi$
in the Néron–Severi group of $X$. We consider the group 
$\mathrm{Aut}_\pi(X)$ consisting of automorphisms of $X$ 
preserving the elliptic fibration. The following
filtration of $\NS(X)$ is preserved by $\mathrm{Aut}_\pi(X)$:
\[
0  \subsetneq [E] \subsetneq [E]^\perp \subsetneq \NS(X),
\]
where $[E]^\perp$ is the orthogonal complement of $[E]$
with respect to the intersection form in $\NS(X)$. By
the 
Hodge index theorem, the intersection form on
the quotient lattice $L \coloneqq [E]^\perp/[E]$ is strictly negative definite. 

Let $\mathrm{Isom}(L)$ be the 
group of automorphisms of the lattice $L$ that preserve
the induced intersection form. Since the induced intersection form
is strictly negative-definite, $\mathrm{Isom}(L)$ is a 
discrete subgroup of the compact orthogonal group $O(\rank(L))$.
Thus, $\mathrm{Isom}(L)$ is finite. 

Since $\mathrm{Aut}_\pi(X)$ preserves the lattice $L$ equipped
with the intersection form, we have the induced group homomorphism:
\[
\rho \colon \mathrm{Aut}_\pi(X) \to \mathrm{Isom}(L).
\]
We define $\ker(\rho) \eqqcolon \bm{\mathrm{Aut}^\circ_\pi(X)}$, 
which is the subgroup of $\mathrm{Aut}_\pi(X)$ whose induced action 
on $[E]^\perp/[E]$ is trivial. 
Because $\mathrm{Isom}(L)$ is 
finite, $\mathrm{Aut}^\circ_\pi(X)$ is a finite index 
subgroup of $\mathrm{Aut}_\pi(X)$.

If we work inside $\mathrm{H}^{1,1}(X, \R)$, we have a 
similar construction of $\mathrm{Aut}^\circ_\pi(X)$.
\begin{lemma}
    \label[lemma]{lem:the_finite_index_subgroup}
    Let $[E]^\perp_{\mathrm{H}^{1,1}}$ denote the orthogonal complement of $[E]$
    in $\mathrm{H}^{1,1}(X, \R)$. Then the following 
    filtration is preserved by $\mathrm{Aut}_\pi(X)$:
    \[
    0  \subsetneq [E]_{\R} \subsetneq [E]^\perp_{\mathrm{H}^{1,1}} 
    \subsetneq \mathrm{H}^{1,1}(X, \R).
    \]

    We define $\bm{\mathrm{Aut}^{\circ, (1, 1)}_\pi(X)}$ as
    the subgroup of $\mathrm{Aut}_\pi(X)$ that acts
    trivially on the quotient space 
    $[E]^\perp_{\mathrm{H}^{1,1}} / [E]_{\R}$. Then 
    $\mathrm{Aut}^{\circ, (1, 1)}_\pi(X)$ is a finite 
    index subgroup of $\mathrm{Aut}_\pi(X)$.
\end{lemma}
\begin{proof}
    Since the intersection form is non-degenerate, we have 
    the following orthogonal decomposition
    \[
    \mathrm{H}^{1,1}(X, \R) = \NS(X)_\R \oplus T^{1,1}(X)_\R,
    \]
    where $T^{1,1}(X)_\R$ is the degree $(1, 1)$-part of
    the real 
    transcendental classes $T(X)_\R$, and 
    $T(X)$ is the transcendental lattice of $X$ 
    (see \cite{Huybrechts_K3_Book}*{Chapter 3, Lemma 3.1}). 

    Note that $[E] \in \NS(X)$, so it is perpendicular to 
    transcendental classes $T^{1,1}(X)_\R$.
    Thus, $[E]^\perp_{\mathrm{H}^{1,1}}$ has a similar decomposition
    \[
    [E]^\perp_{\mathrm{H}^{1,1}} = [E]^\perp_\R \oplus 
    T^{1,1}(X)_\R,
    \]
    which descends to the quotient space as:
    \[
    [E]^\perp_{\mathrm{H}^{1,1}} / [E]_{\R} = 
    ([E]^\perp/[E])_\R \oplus T^{1,1}(X)_\R.
    \]
    
    Since $\mathrm{Aut}_\pi(X)$ preserves the 
    transcendental lattice $T(X)$ with its Hodge structure, 
    the image of 
    the induced group homomorphism:
    \[
    \rho_T \colon \mathrm{Aut}_\pi(X) \to \mathrm{Isom}(T(X)_\R)
    \]
    is a finite group by 
    \cite{Huybrechts_K3_Book}*{Chapter 3, Corollary 3.4}.
    Thus $\ker(\rho_T)$ is a finite index subgroup of 
    $\mathrm{Aut}_\pi(X)$.

    Since $T^{1,1}(X)_\R \subset T(X)_\R$, we have 
    \[
    \ker(\rho_T) \cap \mathrm{Aut}^\circ_\pi(X) 
    \subset \mathrm{Aut}^{\circ, (1, 1)}_\pi(X).
    \]

    Since $\ker(\rho_T)$ and 
    $\mathrm{Aut}^\circ_\pi(X)$ are finite 
    index subgroups, $\mathrm{Aut}^{\circ, (1, 1)}_\pi(X)$
    is also a finite index subgroup of $\mathrm{Aut}_\pi(X)$.
\end{proof}

\begin{remark}
\Cref{lem:the_finite_index_subgroup} also holds for general projective elliptic surfaces. To see this, it suffices to show that the image of the induced homomorphism $\rho_T \colon \mathrm{Aut}_\pi(X) \to \mathrm{Isom}(T(X)_\R)$ is a finite group. This follows from the fact that $\mathrm{Aut}_\pi(X)$ preserves the integral lattice $T(X)$ along with its Hodge structure, and the intersection form on the $(1,1)$-part $T^{1,1}(X)_\R$ is strictly negative definite. (See also the proof of \cite{Huybrechts_K3_Book}*{Chapter 3, Corollary 3.4}.)
\end{remark}

Let $r$ be the index of $\mathrm{Aut}^{\circ, (1, 1)}_\pi(X)$ inside 
$\mathrm{Aut}_\pi(X)$. By elementary group theory, we have $T^{r!} \in \mathrm{Aut}^{\circ, (1, 1)}_\pi(X)$ for any $T \in \mathrm{Aut}_\pi(X)$.

Lastly, we introduce a useful formula on cohomology classes,
which follows directly from the definition of $\mathrm{Aut}^{\circ, (1, 1)}_\pi(X)$.
\begin{lemma}
    \label[lemma]{lem:cohomo_classes_formula}
    Let $v$ be a real $(1, 1)$-class in the cohomology group 
    $\mathrm{H}^{1,1}(X)$.
    Suppose $v \cdot [E] = 0$. 
    Then for any $T \in \mathrm{Aut}^{\circ, (1, 1)}_\pi(X)$, we have 
    \[
    T_*v - v = c(v, T)[E],
    \]
    where $c(v, T) \in \R$ depends linearly on $v$.
\end{lemma}

\subsection{Current-valued pairing}
We recall the construction of \emph{current-valued pairing}
from \cite{Filip-Tosatti21}*{Theorem 3.2.14}. We relate it to 
the \emph{Archimedean height pairing of smooth forms} in \Cref{dfn:pairing_of_forms}.

Let \(\alpha\) be a closed smooth real \((1,1)\)-form on \(X\) such that 
$\int_{X_s} \alpha =0 $ for $s \neq 0$. Let \(\vphi \in C^{\infty}(X^\circ)\) be the \emph{preferred potential} of $\alpha$.

Cohomologically, we have $[\alpha]\cdot [E] = 0$. Using 
\Cref{lem:cohomo_classes_formula}, we obtain 
\[
T_* [\alpha] - [\alpha] = c [E],
\]
for $T \in \mathrm{Aut}^{\circ, (1, 1)}_\pi(X)$.

Note that $[E]$ can be represented by the smooth form 
$\pi^* \beta$, where $\beta$ is a suitable smooth form on the
base $B = \P^1$. We can take an \(f\in C^{\infty}(X)\) such that
\begin{equation}
    \label{eqn:pull-back-ddc}
    T_*\alpha - \alpha = \pi^* \beta + \ddc f.
\end{equation}

Then $f + T_* \vphi - \vphi$ is constant on each $X_s$ for 
$s \in B^\circ$. Therefore, there is a function $u$ on $B^\circ$ such that 
\begin{equation}
    \label{eqn:pullback_fib_constant}
    \pi^* u = f + T_* \vphi - \vphi.
\end{equation}

Let $\omega$ be a \Kahler metric on $X$ normalized 
such that
$[\omega] \cdot [E] = 1$. Then
$u(s)$ can be determined by 
\begin{equation}
    \label{eqn:base_potential_defn}
\int_{X_s} (\pi^* u) \omega_s = u(s) = 
\int_{X_s}(f + T_* \vphi - \vphi)\omega_s
\end{equation}

We call $u$ the \emph{limit potential}. We rewrite
\Cref{eqn:base_potential_defn} 
using the \emph{Archimedean height pairing} 
in \Cref{dfn:pairing_of_forms}:
\begin{equation}
    \label{eqn:relation_height_pairing_and_current_pairing}
    u(s) = \int_{X_s}(f + T_* \vphi - \vphi) \omega_s =
    \int_{X_s} f \omega_s + \la \alpha, (T^* - I)\omega\ra
\end{equation}

The \emph{current-valued pairing} is
defined as $\eta_B(T, [\alpha])\coloneqq \beta + \ddc u$,
which is a current on $B$. It is smooth on $B^\circ$ since
$u \in C^\infty(B^\circ)$.

We now generalize Filip--Tosatti's current-valued pairing
by allowing arbitrary
singular fibers of the elliptic fibration $\pi\colon
X \to B$.
\begin{proposition}
    \label[proposition]{prop:continuity_of_current_val_pairing}
    Let $T \in \mathrm{Aut}^{\circ, (1, 1)}_\pi(X)$, where 
    $\mathrm{Aut}^{\circ, (1, 1)}_\pi(X)$ is a finite index 
    subgroup of $\mathrm{Aut}_\pi(X)$ defined in 
    \Cref{lem:the_finite_index_subgroup}. 
    Then the \emph{current-valued pairing}
    $\eta_{B}(T, [\alpha])$ has a continuous
    potential for any closed smooth real \((1,1)\)-form 
    $\alpha$ on \(X\) such that $\int_{X_s} \alpha =0 $ for $s \neq 0$.
\end{proposition}
\begin{proof}
    We first show that $u \in L^\infty(B^\circ)$
    following Tate's argument in \cite{Tate1983}
    (see also \cite{Filip-Tosatti21}*{Remark 3.2.14}).
    
    From \Cref{eqn:pullback_fib_constant}, for any integer $i \geq 0$, we have 
    \[
    T^i_* \pi^* u = \pi^* u = T^i_*(f + T_*\vphi - \vphi)
    = T^i_* f + (T^{i+1}_* - T_*)\vphi.
    \]
    Then for $k \geq 1$ the Birkhoff sum 
    $S_k(T, f) = \sum_{i= 0}^{k-1}T^i_* f$ satisfies 
    \[
    S_k(T, f) = k \pi^* u + \vphi - T^k_*\vphi.
    \] 
    Thus, for $s \in B^\circ$ and every $k \geq 1$, we have 
    \[
    \abs{u(s)} \leq \frac{1}{k}\sup_{X_s}\abs{S_k(T, f)}
    + \frac{1}{k}\sup_{X_s} \abs{\vphi - T^k_*\vphi}
    \leq \norm{f}_{L^\infty(X)} + \frac{2}{k}
    \norm{\vphi}_{L^{\infty}(X_s)}.
    \]
    Fixing $s \in B^\circ$ and letting $k \to \infty$, we have 
    $\abs{u(s)} \leq \norm{f}_{L^\infty(X)}$. Therefore, 
    $u \in L^\infty(B^\circ)$.

    On the other hand, from the relation between the potential $u$
    and the \emph{Archimedean height pairing} in 
    \Cref{eqn:relation_height_pairing_and_current_pairing},
    the Archimedean height pairing
    \[
    \la \alpha, (T^* - I)\omega\ra(s) = u(s) - \int_{X_s} f\omega_s
    \in L^\infty(B^\circ)
    \]
    is bounded.
    
    By \Cref{cor:continuity_of_pairing_general_fiber}, 
    $\la \alpha, (T^* - I)\omega\ra$ can be extended
    continuously on $B$ because it is bounded and has no 
    singularity. Using 
    Barlet's result in \Cref{thm:fib_integral_barlet},
    the function $s \mapsto \int_{X_s} f\omega_s$
    is Hölder continuous on $B$.

    In summary, $u(s) = \la \alpha, (T^* - I)\omega\ra(s) + 
    \int_{X_s} f\omega_s
    $ is continuous on $B$. 
    So  $\eta_{B}(T, [\alpha])$ has a continuous
    potential.
\end{proof}

As a corollary, we have 
\begin{corollary}
    \label[corollary]{cor:dynamical_consequences}
    Let $\omega$ be a \Kahler metric on $X$. For any 
    $T \in \mathrm{Aut}^{\circ, (1, 1)}_\pi(X)$, the limit current 
    $$
        \lim_{n \to \infty} \frac{(T^n)_* \omega}{n^2} = \frac{1}{2} \pi^*\eta_B(T,
    [T_*\omega - \omega]).
    $$ has a 
    continuous potential. 

    We note that the convergence of currents $\frac{(T^n)_* \omega}{n^2} \to \frac{1}{2}\pi^*\eta_B(T, [T_*\omega - \omega])$ cannot be in $C^0_{\rm loc}(X\setminus X_0)$.
    
    In particular, let $\omega$ be a Ricci flat \Kahler metric on the K3 surface $X$, i.e., 
    \[
    \omega^2 = \vol(\omega) \Omega \wedge \overline{\Omega},
    \]
    where $\Omega$ is the normalized holomorphic $(2, 0)$-form on 
    $X$. Then $T^n_* \omega$ is also Ricci flat. 
    Therefore,
    $\{[\frac{T^n_* \omega}{n^2}]\}$
    provides a counterexample to 
    Tosatti's question \cite{tosatti2025}*{Question 1.5(b)}.
\end{corollary}
\begin{proof}
    By \cite{Filip-Tosatti21}*{Corollary 3.2.20}, we have 
    \[
    \lim_{n \to \infty} \frac{(T^n)_* \omega}{n^2} = \frac{1}{2} \pi^*\eta_B(T,
    [T_*\omega - \omega]).
    \]
    So $\lim_{n \to \infty} \frac{(T^n)_* \omega}{n^2}$
    has continuous potential by 
    \Cref{prop:continuity_of_current_val_pairing}.

    We now prove that the convergence of currents
    $\frac{(T^n)_* \omega}{n^2}
    \to \frac{1}{2}\pi^*\eta_B(T,
    [T_*\omega - \omega])$ is not a $C^0_{\rm loc}(X\setminus X_0)$
    convergence. We proceed by contradiction, assuming that
    $\frac{(T^n)_* \omega}{n^2} \to \frac{1}{2}\pi^*\eta_B(T,
    [T_*\omega - \omega])$ in $C^0_{\rm loc}(X\setminus X_0)$.

    We let $\xi = T_* \omega - \omega$. Then by elementary
    computation (see also \cite{Filip-Tosatti21}*{Proposition 3.2.15}), 
    we have 
    \[
    T_*^n \omega = \omega + n \xi + \ddc(n \vphi - S_n(T, \vphi))
    + \frac{n(n-1)}{2}\pi^*\eta_B(T, [\xi]),
    \]
    where $\vphi$ is the \emph{preferred potential} of 
    $\xi$ and $S_n(T, \vphi) = \sum_{k=0}^{n-1}T_*^k \vphi  $ is a Birkhoff sum.

    Note that $\frac{1}{n^2}(\omega + n \xi + \frac{n(n-1)}{2}\pi^*\eta_B(T, [\xi]))$ converges to $\frac{1}{2} \pi^*\eta_B(T,
    [\xi])$ in $C^0_{\rm loc}(X\setminus X_0)$.
    Under our assumption for contradiction, we have 
    \begin{equation}
        \label{eqn:non-convergence-1}
        \frac{1}{n^2}\ddc(n \vphi - S_n(T, \vphi)) \to_{n \to\infty} 0 \quad \text{in }
    C^0_{\rm loc}(X\setminus X_0).
    \end{equation}

    It suffices to show that \Cref{eqn:non-convergence-1}
    is impossible. 
    We interpret $\vphi$, the \emph{preferred potential} of $\xi = T_* \omega - \omega$, as follows: Since $\pi$ is an elliptic fibration, each smooth fiber $X_s$ is isomorphic to a torus $\mathbb{C}/\Lambda_s$, and $T$ acts on $X_s$ by translations.
    Let $\omega_{\rm{flat}, s}$ be the 
    normalized flat metric on $X_s$. There exists $\rho_s \in C^\infty(X_s)$
    such that 
    \[
    \omega_{\rm{flat}, s} - \omega_s = \ddc \rho_s.
    \]
    Applying $T_*$ to this equation, we obtain 
    $\omega_{\rm{flat}, s} - T_* \omega_s = \ddc T_* \rho_s$,
    since $\omega_{\rm{flat}, s}$ is invariant 
    under translation $T$. Then we have 
    \[
    \xi\mid_{X_s} = T_* \omega_s - \omega_s = -\ddc (T_* \rho_s-\rho_s).
    \]

    Thus, the \emph{preferred potential} $\vphi$ of 
    $\xi$ is equal to $T_* \rho_s-\rho_s$ up to a constant on $X_s$.
    By Ehresmann's theorem and solving equations on the region where the fibration is locally trivial, there exists $\rho \in C^\infty(X\setminus X_0)$ such that $\rho\mid_{X_s} = \rho_s$ on the smooth fibers $X_s$.
    Furthermore, $\vphi =  T_* \rho-\rho + c(s)$, 
    where $c(s)$
    is a smooth function on $B^\circ$.

    Substituting $\vphi$ by $T_* \rho-\rho + c(s)$
    in \Cref{eqn:non-convergence-1}, we have 
    \[
    \frac{1}{n^2}\ddc(n(T_* \rho-\rho) - (T^n_* \rho-\rho)) \to_{n \to\infty} 0 \quad \text{in }
    C^0_{\rm loc}(X\setminus X_0).
    \]
    By smoothness of $\rho$ on $X \setminus X_0$, the above convergence
    is equivalent to 
    \begin{equation}
        \label{eqn:non-convergence-2}
        \frac{1}{n^2}\ddc (T^n_* \rho) \to_{n \to\infty} 0 \quad \text{in }
        C^0_{\rm loc}(X\setminus X_0).
    \end{equation}
    We will show that \Cref{eqn:non-convergence-2} 
    is impossible. 
    
    On a small open set $U \subset B^\circ$
    not containing any singular points, we can trivialize
    the fibration $\pi$ as $U \times \C/\Lambda \to U$. Moreover,
    we can lift the $U \times \C/\Lambda$ to its universal covering
    $U \times \C$, and the fiber-preserving automorphism $T$
    can also be lifted, yielding the following commutative diagram:
    $$
    \begin{tikzcd}
        U_s \times \mathbb{C}_z \arrow[r, "p"] 
                            \arrow[dr, "pr_1"'] 
                            \arrow[loop left, "\tilde{T}"] 
        & \pi^{-1}(U_s) \arrow[d, "\pi"] 
                    \arrow[loop right, "T"] \\
        & U_s
    \end{tikzcd}
    $$
    We can then express the action of $\tilde{T}$ on the covering space $U_s \times \mathbb{C}_z$ as $\tilde{T}(s, z) = (s, z + T(s))$, where $T \colon U_s \to \mathbb{C}$ is a holomorphic function.
    Following the argument in \cite{Cantat2001}*{Proposition 2.2} (see also \cite{filip2022k3}*{Proposition 6.2.7}, \cite{Cantat-Dujardin-2023-B}*{Lemma 3.7}), 
    the function $T$ is non-constant.

    Lifting the convergence in \Cref{eqn:non-convergence-2} to the covering space, we find it is equivalent to
    \begin{equation}
        \label{eqn:non-convergence-3}
        \frac{1}{n^2}\ddc (\tilde{T}^n_* \tilde{\rho}) \to_{n \to\infty} 0 \quad \text{in }
        C^0_{\rm loc}(U_s \times \C_z),
    \end{equation}
    for all small open neighborhoods of points in $B^\circ$ and its 
    corresponding covering, where $\tilde{\rho}$
    is the lift of $\rho$ to the covering.

    Note that $\tilde{T}^n_* \tilde{\rho} 
    = \tilde{\rho}(s, z - n T(s))$. By the holomorphic chain rule, 
    we have 
    \begin{align*}
        \pa_s\db_s (\tilde{T}^n_* \tilde{\rho})(s,z)&= \pa_s\db_s(\tilde{\rho}(s, z - n T(s)))\\
    &=\pa_s((\db_s \tilde{\rho})(s, z - n T(s)) - n\overline{ \pa_s T}(\db_z \tilde{\rho})(s, z - nT(s))) \quad (\db_s T(s) = 0)\\
    &=(\pa_s\db_s \tilde{\rho})(s, z - n T(s)) - n (\pa_s T(s)) (\pa_z \db_s \tilde{\rho})(s, z - n T(s))\\
    &-n\overline{\pa_s T} (\pa_s\db_z \tilde{\rho})(s, z - nT(s)) + {n^2 \abs{\pa_s T}^2(\pa_z \db_z \tilde{\rho})(s, z - n T(s))}
    \end{align*}

    Then \Cref{eqn:non-convergence-3} implies that 
    \[
    \abs{\pa_s T}^2(\pa_z \db_z \tilde{\rho})(s, z - n T(s))
    \to_{n \to\infty} 0 \quad \text{in }
    C^0_{\rm loc}(U_s \times \C_z)
    \]
    for all small open sets in $B^\circ$. 
    
    Since $T$
    is non-constant and holomorphic, we have $\abs{\pa_s T}^2 > 0$
    except for at most finitely many $s \in U_s$. 
    Therefore, there exists a finite set $\Sigma \subset U_s$ such that
    \[
    (\pa_z \db_z \tilde{\rho})(s, z - n T(s))
    \to_{n \to\infty} 0 \quad \text{in }
    C^0_{\rm loc}((U_s \setminus \Sigma) \times \C_z).
    \]
    Since $\tilde{\rho}$ is periodic in $z$, it implies that 
    $\pa_z \db_z \tilde{\rho}(s, z) = 0$ for $s \in U_s \setminus \Sigma$. By smoothness of $\tilde{\rho}$, we have $\pa_z \db_z \tilde{\rho}(s, z) \equiv 0$ for $s \in S^\circ$. Consequently, 
    $\ddc \rho_s = 0$ on $X_s$ for $s \in S^\circ$.

    By definition of $\rho$, we have 
    \[
    \omega_{\rm{flat}, s} - \omega_s = \ddc \rho_s = 0.
    \]
    Thus, $\omega_{\rm{flat}, s} = \omega_s$ for $s \in S^\circ$.

    To get a contradiction, we compute the Gauss curvature of $\omega_{\rm{flat}, s}$ and $\omega$ on smooth fibers $X_s$, and study them when $X_s$ is approaching a singular fiber.
    
    By definition, the Gauss curvature 
    $K_{\omega_{\rm{flat}, s}} \equiv 0$. On the 
    other hand, by \cite{Dai-Yoshikawa2025}*{Lemma 10.1},
    the minimum of $K_{\omega_s}$ goes to $-\infty$ as 
    $s \to 0$. 
    This leads to a contradiction, concluding the proof that the convergence is not in $C^0_{\rm loc}(X\setminus X_0)$.

\end{proof}

\appendix
\section{Asymptotics of a fiber integral}
\label[appendix]{section:6-appendix}

We prove a lemma on the asymptotics of the fiber integral of the 
log-norm of the canonical section of a divisor. This result also 
appears in other works (e.g., \cite{Yoshikawa2006}*{Section 4}); 
however, we provide a proof here in our setting for convenience and 
completeness.

\begin{lemma}
    \label[lemma]{lem:log_norm_section_fib_int_general_case}
    Let $\pi\colon X \to \D$ be a degeneration of algebraic curves.
    Assume $X$ is smooth and we allow a general singular fiber $X_0$
    here. Let $C$ be one of the irreducible 
    components of 
    $X_0$ and $L = \mathcal{O}(C)$ the corresponding line bundle with 
    Hermitian metric $h$.

    Let $\sigma$ be the canonical section of $L$ which vanishes 
    exactly on $C$. Then, for any smooth $(1, 1)$-form $\alpha$ 
    on $X$, there exists a constant $c \in \Q_{>0}$ such that 
    \[
    \int_{X_s}\log\norm{\sigma}^2_{h} \alpha - \left( \int_C \alpha \right)\log  \abs{s}^2
    \]
    extends Hölder continuously through the origin.
\end{lemma}
\begin{proof}
    We consider a semi-stable reduction for the family
    $\pi \colon (X, X_0) \to (\D, \{0\})$ given by the 
    following commutative diagram:
    \[
    \begin{tikzcd}
        Y \arrow[r, "\mu"] \arrow[dr, "p"'] & X \times_{\D} \D_t \arrow[r, "pr_1"] \arrow[d, "pr_2"] & X \arrow[d, "\pi"] \\
        & \D_t \arrow[r, "f_d"] & \D_s
    \end{tikzcd}
    \]
    Here, $f_d(t) = t^d = s$ is a base change of degree $d$, and $\mu$ is a resolution of singularities such that $Y$ is a smooth complex surface. The singular fiber $Y_0 = p^{-1}(X_0)$ is a reduced divisor with simple normal crossings. Let $\tilde{\mu} = pr_1 \circ \mu$.

    Let $m$ be the multiplicity of $C$ in $X_0$. Since $\tilde{\mu}^* X_0 = d Y_0$ and $Y_0$ is reduced, the strict transform $\tilde{C}$ of $C$ appears in $Y_0$ with multiplicity $1$. Thus, the pullback of $C$ as a Cartier divisor is
    \[
    \tilde{\mu}^*C = \frac{d}{m} \tilde{C} + \sum_{k} \nu_k E_k
    \]
    where $E_k$ are exceptional divisors.
    
    This induces a decomposition of the log-norm on $Y$:
    \begin{equation}
        \label{eqn:tmp_appendix_1}
        \log \norm{\tilde{\mu}^*\sigma}^2_{\tilde{\mu}^*h} = \frac{d}{m} \log \norm{\sigma_{\tilde{C}}}^2_{h_{\tilde{C}}} + \sum_{k} \nu_k \log \norm{\sigma_{E_k}}^2_{h_{E_k}} + u
    \end{equation}
    where $\sigma_D$ denotes the canonical section of a component $D \subset Y_0$, and $u$ is a smooth function on $Y$.

    For $t \neq 0$, the map $\tilde{\mu}_t \colon Y_t \xrightarrow{\sim} X_s$ is an isomorphism. Pulling the integral back to the semi-stable model $Y_t$, we obtain:
    \[
    \int_{X_s} \log\norm{\sigma}^2_h \alpha = \int_{Y_t} \log\norm{\tilde{\mu}^*\sigma}^2_{\tilde{\mu}^*h} \tilde{\mu}^*\alpha
    \]
    Substituting our decomposition \Cref{eqn:tmp_appendix_1}, this splits into:
    \begin{equation}
        \label{eqn:tmp_appendix_2}
        \frac{d}{m} \int_{Y_t} \log\norm{\sigma_{\tilde{C}}}^2_{h_{\tilde{C}}} \tilde{\mu}^*\alpha + \sum_{k} \nu_k \int_{Y_t} \log\norm{\sigma_{E_k}}^2_{h_{E_k}} \tilde{\mu}^*\alpha + \int_{Y_t} u \tilde{\mu}^*\alpha = \int_{X_s} \log\norm{\sigma}^2_h \alpha
    \end{equation}

    Since $Y_0$ has strictly simple normal crossings, we apply
    \Cref{lem:log_norm_section_fib_int_snc_case} to the
    first two terms of \Cref{eqn:tmp_appendix_2}, which yield:
    \begin{align*}
        \int_{Y_t} \log\norm{\sigma_{\tilde{C}}}^2_{h_{\tilde{C}}} \tilde{\mu}^*\alpha &=  \left(\int_{\tilde{C}} \tilde{\mu}^*\alpha \right) \log\abs{t}^2 + R_1(t)\\
        &= m\left(\int_{{C}} \alpha \right) \log\abs{t}^2 + R_1(t)\\
        \int_{Y_t} \log\norm{\sigma_{E_k}}^2_{h_{E_k}} \tilde{\mu}^*\alpha&
        = \left(\int_{E_k} \tilde{\mu}^*\alpha \right) \log\abs{t}^2 + R_{2,k}(t),\\
        &= \left(\int_{\tilde{\mu}_*[E_k]} \alpha \right) \log\abs{t}^2 + R_{2,k}(t)\\ 
        &= R_{2,k}(t),
    \end{align*}
    where $R_1(t), R_{2,k}(t)$ are Hölder continuous around $0$.
    
    Since $u$ is smooth, by Barlet's result (\Cref{thm:fib_integral_barlet}), the third term 
    (of \Cref{eqn:tmp_appendix_2}) $\int_{Y_t} u \tilde{\mu}^*\alpha \eqqcolon R_3(t)$, is Hölder continuous around $0$. Substituting the previous computations into \Cref{eqn:tmp_appendix_2}, we have 
    \[
    \int_{X_s} \log\norm{\sigma}^2_h \alpha = 
    d\left(\int_{{C}} \alpha \right) \log\abs{t}^2 + \underbrace{\frac{d}{m} R_1(t)
    + \sum_{k}R_{2,k}(t) + R_3(t)}_{\text{Hölder continuous}},
    \]
    
    Finally, we translate back to the original base parameter $s$. Since $s = t^d$, we have $\log\abs{t}^2 = \frac{1}{d} \log\abs{s}^2$.  Therefore,
    \[
    \int_{X_s} \log\norm{\sigma}^2_h \alpha = 
    \left( \int_C \alpha \right) \log\abs{s}^{2} + R(s),
    \]
    and $R(s)$ is Hölder continuous.

    We conclude that:
    \[
    \int_{X_s}\log\norm{\sigma}^2_{h} \alpha - \left( \int_C \alpha \right)\log  \abs{s}^2
    \]
    extends Hölder continuously through the origin.
\end{proof}
\vone
\begin{lemma}
    \label[lemma]{lem:log_norm_section_fib_int_snc_case}
    Let $\pi \colon \mathcal{Y} \to \D$ be a proper 
    holomorphic map from a smooth complex surface to 
    the unit disk. Assume the singular fiber 
    $Y_0 = \pi^{-1}(0)$ is a reduced divisor with 
    strictly simple normal crossings. Let $D$ be an 
    irreducible component of $Y_0$, $\sigma_D$ be its 
    canonical section, and $h$ be a smooth Hermitian 
    metric on $\mathcal{O}(D)$. For any smooth 
    $(1,1)$-form $\omega$ on $\mathcal{Y}$, 
    the function 
    \[
    t \mapsto \int_{Y_t} \log\norm{\sigma_D}_h^2 \, \omega - \left( \int_D \omega \right) \log\abs{t}^2
    \]
    extends Hölder continuously to $t=0$.
\end{lemma}

\begin{proof}
    Fix a finite open cover $\{U_j\}$ of $Y_0$ in 
    $\mathcal{Y}$ by coordinate polydisks 
    $U_j \cong \D^2$ such that on each $U_j$, 
    $\pi$ is given either by $t = z_1$ 
    (smooth points of $Y_0$) or $t = z_1 z_2$ (nodes of $Y_0$). 
    Let $\{\rho_j\}$ be a smooth partition of unity subordinate to $\{U_j\}$. 

    On each $U_j$, we decompose the metric as 
    $\log\norm{\sigma_D}_h^2 = \log\abs{f_D}^2 + u_j$, 
    where $f_D \in \mathcal{O}(U_j)$ defines $D \cap U_j$ and $u_j$ is a smooth function. 
    By Barlet's theorem on fiber integral (\Cref{thm:fib_integral_barlet}), 
    $t \mapsto \int_{Y_t \cap U_j} \rho_j u_j \omega$ 
    extends Hölder continuously to $t=0$. We analyze 
    the singular integral $I_j(t) = \int_{Y_t \cap U_j} \rho_j \log\abs{f_D}^2 \omega$ case by case:

    \textbf{Case 1: $U_j \cap D = \emptyset$.} 
    Here, $f_D$ is a non-vanishing holomorphic 
    function on $U_j$. Thus $\log\abs{f_D}^2$ is 
    smooth, and $I_j(t)$ is smooth around
    $t = 0$. In this case,
    \[
    I_j(t) = \left( \int_{D \cap U_j} \rho_j \omega \right) \log\abs{t}^2 + I_j(t).
    \]

    \textbf{Case 2: $U_j$ intersects $D$ outside the nodes.} 
    Here, $\pi = z_1$ and $f_D = z_1$. 
    Thus $I_j(t) = \log\abs{t}^2 \int_{Y_t \cap U_j} \rho_j \omega$. 
    Since $\pi$ on $U_j$
    is a submersion, by Ehresmann's theorem and
    Taylor expansion, we have 
    $\int_{Y_t  \cap U_j} \rho_j \omega = \int_{\{z_1 = t\}  \cap U_j} \rho_j \omega = \int_{\{z_1 = 0\}  \cap U_j}\rho_j \omega + O(\abs{t}) =   \int_{D \cap U_j} \rho_j \omega + O(\abs{t})$. 
    Multiplying by $\log\abs{t}^2$, we obtain
    \[
    I_j(t) = \left( \int_{D \cap U_j} \rho_j \omega \right) \log\abs{t}^2 + R_j(t),
    \]
    where $R_j(t) = O(\abs{t}\log\abs{t})$ is Hölder 
    continuous around $t=0$.

\textbf{Case 3: $U_j$ contains a node $D \cap D'$.}
    Here, the map is given by $\pi = z_1 z_2$, with $D$ defined by $z_1 = 0$ and $D'$ defined by $z_2 = 0$. Thus $f_D = z_1$. The fiber $Y_t \cap U_j$ is the restricted annulus $\{ (z_1, z_2) \in U_j \mid z_1 z_2 = t \}$.

    Let $\eta = \rho_j \omega$. Since $\eta$ is a smooth form with compact support in $U_j \cong \D^2$, we can write it globally on the chart as $\eta = \frac{i}{2} \sum_{k,l=1}^2 \eta_{k\bar{l}} dz_k \wedge d\bar{z}_l$. We want to evaluate the asymptotic behavior of
    \[
    I_j(t) = \int_{Y_t \cap U_j} \log\abs{z_1}^2 \, \eta
    \]

    We partition the domain of integration on the annulus $Y_t \cap U_j$ into two regions: $V_1(t) = \{ \abs{t}^{1/2} \le \abs{z_1} < 1 \}$ and $V_2(t) = \{ \abs{t}^{1/2} < \abs{z_2} < 1 \}$. 

    On $V_1(t)$, we use $z_1$ as the coordinate. Substituting $z_2 = t/z_1$ and $dz_2 = -(t/z_1^2) dz_1$, the restriction of $\eta$ is:
    \[
    \eta|_{Y_t} = \left( \eta_{1\bar{1}} - \eta_{1\bar{2}} \frac{\bar{t}}{\bar{z}_1^2} - \eta_{2\bar{1}} \frac{t}{z_1^2} + \eta_{2\bar{2}} \frac{\abs{t}^2}{\abs{z_1}^4} \right) \frac{i}{2} dz_1 \wedge d\bar{z}_1
    \]
    where the smooth coefficients $\eta_{k\bar{l}}$ are evaluated at $(z_1, t/z_1)$. By the smoothness of $\eta$, we can Taylor expand the leading coefficient as $\eta_{1\bar{1}}(z_1, t/z_1) = \eta_{1\bar{1}}(z_1, 0) + O(\abs{t}/\abs{z_1})$.

    Integrating the cross-term and higher-order errors against $\log\abs{z_1}^2$ over $V_1(t)$ yields bounded remainders. For example, passing to polar coordinates $r = \abs{z_1}$:
    \[
    \int_{\abs{t}^{1/2}}^1 \abs{\log r} \cdot O\left( \frac{\abs{t}}{r^2} \right) r \, dr = O\left( \abs{t} \int_{\abs{t}^{1/2}}^1 \frac{\abs{\log r}}{r} dr \right) = O(\abs{t} \log^2\abs{t})
    \]
    Similar bounds hold for the $O(\abs{t}^2/r^4)$ and $O(\abs{t}/r)$ terms. Thus, the integral over $V_1(t)$ is tightly dominated by the leading term:
    \[
    \int_{V_1(t)} \log\abs{z_1}^2 \eta|_{Y_t} = \int_{\abs{t}^{1/2} \le \abs{z_1} < 1} \log\abs{z_1}^2 \eta_{1\bar{1}}(z_1, 0) \frac{i}{2} dz_1 \wedge d\bar{z}_1 + O(\abs{t} \log^2\abs{t})
    \]

    Extending the domain of this integral to the full local component $D' \cap U_j$ (where $z_2 = 0$) introduces an error bounded by the integral over the missing inner disk $\{ \abs{z_1} < \abs{t}^{1/2} \}$. This hole has a volume integral of $\int_0^{\abs{t}^{1/2}} \abs{\log r} r \, dr = O(\abs{t} \log\abs{t})$. Hence:
    \begin{align*}
        I_{j,1}(t) &= \underbrace{\int_{D' \cap U_j} \log\abs{z_1}^2 \eta|_{D'}}_{\mathrm{constant}} + O(\abs{t} \log^2\abs{t}).
    \end{align*}

    On $V_2(t)$, we switch to the coordinate $z_2$. Here, the log-norm splits exactly as $\log\abs{z_1}^2 = \log\abs{t}^2 - \log\abs{z_2}^2$.

    Similarly, the restriction of $\eta$ in the $z_2$ coordinate is:
    \[
    \eta|_{Y_t} = \left( \eta_{2\bar{2}} - \eta_{1\bar{2}} \frac{t}{z_2^2} - \eta_{2\bar{1}} \frac{\bar{t}}{\bar{z}_2^2} + \eta_{1\bar{1}} \frac{\abs{t}^2}{\abs{z_2}^4} \right) \frac{i}{2} dz_2 \wedge d\bar{z}_2
    \]
    with coefficients evaluated at $(t/z_2, z_2)$.

    The term involving $-\log\abs{z_2}^2$ behaves exactly like the $V_1(t)$ integral, yielding a constant relative to $t$ minus a remainder:
    \[
    -\int_{V_2(t)} \log\abs{z_2}^2 \eta|_{Y_t} = \underbrace{-\int_{D \cap U_j} \log\abs{z_2}^2 \eta|_{D}}_{\mathrm{constant}} + O(\abs{t} \log^2\abs{t})
    \]

    For the $\log\abs{t}^2$ term, we must integrate the volume-like form $\eta|_{Y_t}$ over $V_2(t)$. Applying the exact same Taylor expansion and error bounds, the dominant term is $\eta_{2\bar{2}}(0, z_2)$. We have:
    \[
    \int_{V_2(t)} \eta|_{Y_t} = \int_{\abs{t}^{1/2} \le \abs{z_2} < 1} \eta_{2\bar{2}}(0, z_2) \frac{i}{2} dz_2 \wedge d\bar{z}_2 + \int_{\abs{t}^{1/2}}^1 O\left( \frac{\abs{t}}{r^2} \right) r \, dr
    \]
    The error integral evaluates to $O(\abs{t} \abs{\log\abs{t}})$. The difference between the main integral and the full integral over $D \cap U_j$ is the missing disk $\{ \abs{z_2} < \abs{t}^{1/2} \}$, which adds a simple $O(\abs{t})$ error. Thus:
    \[
    \int_{V_2(t)} \eta|_{Y_t} = \int_{D \cap U_j} \eta|_{D} + O(\abs{t} \log\abs{t})
    \]
    Multiplying this volume result by $\log\abs{t}^2$ gives $\left( \int_{D \cap U_j} \rho_j \omega \right) \log\abs{t}^2 + O(\abs{t} \log^2\abs{t})$.

    Combining both regions for the entire chart $U_j$, we have
    \[
    I_j(t) = \left( \int_{D \cap U_j} \rho_j \omega \right) \log\abs{t}^2 + \tilde{R}_j(t),
    \]
    where $\tilde{R}_j(t) = \mathrm{constant} + O(\abs{t} \log^2\abs{t})$ is Hölder continuous for any exponent $0 < \alpha < 1$ around $t = 0$.
\end{proof}

\bibliographystyle{amsalpha}
\bibliography{refs}

@article{Tosatti2010,
  title = {Adiabatic limits of {Ricci-flat K\"{a}hler} metrics},
  volume = {84},
  ISSN = {0022-040X},
  url = {http://dx.doi.org/10.4310/jdg/1274707320},
  DOI = {10.4310/jdg/1274707320},
  number = {2},
  journal = {Journal of Differential Geometry},
  publisher = {International Press of Boston},
  author = {Tosatti,  Valentino},
  year = {2010},
  month = feb 
}

@book{Duistermaat2010,
  author    = {Johannes Jisse Duistermaat},
  title     = {Discrete Integrable Systems: {QRT Maps} and {Elliptic Surfaces}},
  series    = {Springer Monographs in Mathematics},
  publisher = {Springer},
  address   = {New York, NY},
  year      = {2010},
  doi       = {10.1007/978-0-387-72923-7},
  isbn      = {978-0-387-72923-7}
}

@article{Cantat-Dujardin-2023-B,
  title = {Invariant {Measures} for {Large} {Automorphism} {Groups} of {Projective} {Surfaces}},
  volume = {30},
  ISSN = {1531-586X},
  url = {http://dx.doi.org/10.1007/s00031-022-09782-0},
  DOI = {10.1007/s00031-022-09782-0},
  number = {1},
  journal = {Transformation Groups},
  publisher = {Springer Science and Business Media LLC},
  author = {Cantat,  Serge and Dujardin,  Romain},
  year = {2023},
  month = jan,
  pages = {75--145}
}

@article{Silverman1,
  author  = {Silverman, Joseph H.},
  title   = {Variation of the canonical height on elliptic surfaces. {I}. {T}hree examples},
  journal = {J. Reine Angew. Math.},
  volume  = {426},
  year    = {1992},
  pages   = {151--178},
  doi     = {10.1515/crll.1992.426.151},
  mrnumber= {1153245}
}

@article{Silverman2,
  author  = {Silverman, Joseph H.},
  title   = {Variation of the canonical height on elliptic surfaces. {II}. {L}ocal analyticity properties},
  journal = {J. Number Theory},
  volume  = {48},
  number  = {3},
  year    = {1994},
  pages   = {291--329},
  doi     = {10.1006/jnth.1994.1069},
  mrnumber= {1293863}
}

@article{Silverman3,
  author  = {Silverman, Joseph H.},
  title   = {Variation of the canonical height on elliptic surfaces. {III}. {G}lobal boundedness properties},
  journal = {J. Number Theory},
  volume  = {48},
  number  = {3},
  year    = {1994},
  pages   = {330--352},
  doi     = {10.1006/jnth.1994.1070},
  mrnumber= {1293864}
}

@article{Cantat-Dujardin-2023,
  title = {Random dynamics on real and complex projective surfaces},
  volume = {802},
  pages = {1--76},
  ISSN = {1435-5345},
  url = {https://doi.org/10.1515/crelle-2023-0038},
  DOI = {10.1515/crelle-2023-0038},
  journal = {Journal f\"{u}r die reine und angewandte Mathematik (Crelles Journal)},
  publisher = {Walter de Gruyter GmbH},
  author = {Cantat, Serge and Dujardin, Romain},
  year = {2023},
  month = aug
}

@article{DeMarcoMavraki2020,
  author  = {DeMarco, Laura and Mavraki, Niki Myrto},
  title   = {Variation of canonical height and equidistribution},
  journal = {American Journal of Mathematics},
  volume  = {142},
  number  = {2},
  pages   = {443--473},
  year    = {2020},
  doi     = {10.1353/ajm.2020.0012},
  publisher = {Johns Hopkins University Press}
}

@article{EFG2018,
  author    = {Eriksson, Dennis and Freixas i Montplet, Gerard and Mourougane, Christophe},
  title     = {Singularities of metrics on {H}odge bundles and their topological invariants},
  journal   = {Algebraic Geometry},
  volume    = {5},
  number    = {6},
  pages     = {742--775},
  year      = {2018},
  doi       = {10.14231/AG-2018-021},
  publisher = {Foundation Compositio Mathematica}
}

@article{Takayama2020,
  title = {Asymptotic expansions of fiber integrals over higher-dimensional bases},
  volume = {2021},
  ISSN = {1435-5345},
  url = {http://dx.doi.org/10.1515/crelle-2020-0027},
  DOI = {10.1515/crelle-2020-0027},
  number = {773},
  journal = {Journal f\"{u}r die reine und angewandte Mathematik (Crelles Journal)},
  publisher = {Walter de Gruyter GmbH},
  author = {Takayama,  Shigeharu},
  year = {2020},
  month = sep,
  pages = {67--128}
}

@article{Tate1983,
  title = {Variation of the {Canonical Height of a Point Depending on a Parameter}},
  volume = {105},
  ISSN = {0002-9327},
  url = {http://dx.doi.org/10.2307/2374389},
  DOI = {10.2307/2374389},
  number = {1},
  journal = {American Journal of Mathematics},
  publisher = {JSTOR},
  author = {Tate,  J.},
  year = {1983},
  month = feb,
  pages = {287}
}

@incollection{filip2022k3,
  author    = {Filip, Simion},
  title     = {An introduction to {K3} surfaces and their dynamics},
  booktitle = {Teichm{\"u}ller theory and dynamics},
  series    = {Panoramas et Synth{\`e}ses},
  year      = {2022},
  pages     = {1--43},
  publisher = {Soci{\'e}t{\'e} Math{\'e}matique de France},
  address   = {Paris},
  note      = {Panoramas et Synthèses. \url{https://math.uchicago.edu/~sfilip/public_files/lectures_k3_dynamics.pdf}}
}

@book{BHPVdV2004,
  author    = {Barth, Wolf P. and Hulek, Klaus and Peters, Chris A. M. and Van de Ven, Antonius},
  title     = {Compact Complex Surfaces},
  series    = {Ergebnisse der Mathematik und ihrer Grenzgebiete. 3. Folge},
  volume    = {4},
  edition   = {2nd},
  publisher = {Springer-Verlag},
  address   = {Berlin},
  year      = {2004},
  doi       = {10.1007/978-3-642-57739-0},
  isbn      = {978-3-540-00832-3}
}

@article{Cantat2001,
  title = {Sur la dynamique du groupe d’automorphismes des surfaces {K3}},
  volume = {6},
  ISSN = {1531-586X},
  url = {http://dx.doi.org/10.1007/BF01263089},
  DOI = {10.1007/bf01263089},
  number = {3},
  journal = {Transformation Groups},
  publisher = {Springer Science and Business Media LLC},
  author = {Cantat,  S.},
  year = {2001},
  month = sep,
  pages = {201--214}
}

@misc{tosatti2025,
      title={Ricci-flat metrics on {Calabi-Yau} manifolds}, 
      author={Valentino Tosatti},
      year={2025},
      eprint={2509.25607},
      archivePrefix={arXiv},
      primaryClass={math.DG},
      url={https://arxiv.org/abs/2509.25607}, 
      note = {Preprint. \href{https://arxiv.org/abs/2509.25607v2}{arXiv:2509.25607v2}}
}

@book{Petersen2016,
  title = {Riemannian {Geometry}},
  ISBN = {9783319266541},
  ISSN = {2197-5612},
  url = {http://dx.doi.org/10.1007/978-3-319-26654-1},
  DOI = {10.1007/978-3-319-26654-1},
  journal = {Graduate Texts in Mathematics},
  publisher = {Springer International Publishing},
  author = {Petersen,  Peter},
  year = {2016}
}

@article{Yoshikawa2006,
  title = {On the singularity of {Quillen} metrics},
  volume = {337},
  ISSN = {1432-1807},
  url = {http://dx.doi.org/10.1007/s00208-006-0027-5},
  DOI = {10.1007/s00208-006-0027-5},
  number = {1},
  journal = {Mathematische Annalen},
  publisher = {Springer Science and Business Media LLC},
  author = {Yoshikawa,  Ken-Ichi},
  year = {2006},
  month = jul,
  pages = {61--89}
}

@article{DH2015,
  title = {Asymptotics of the {Néron} height pairing},
  volume = {22},
  ISSN = {1945-001X},
  url = {http://dx.doi.org/10.4310/MRL.2015.v22.n5.a5},
  DOI = {10.4310/mrl.2015.v22.n5.a5},
  number = {5},
  journal = {Mathematical Research Letters},
  publisher = {International Press of Boston},
  author = {Holmes,  David and de Jong,  Robin},
  year = {2015},
  pages = {1337--1371}
}

@misc{GPSS-Sobolev,
      title={Sobolev inequalities on {K\"ahler} spaces}, 
      author={Bin Guo and Duong H. Phong and Jian Song and Jacob Sturm},
      year={2023},
      eprint={2311.00221},
      archivePrefix={arXiv},
      primaryClass={math.DG},
      url={https://arxiv.org/abs/2311.00221}, 
      note = {Preprint. \href{https://arxiv.org/abs/2311.00221v1}{arXiv:2311.00221v1}}
}

@article{Barlet,
  title = {Développement asymptotique des fonctions obtenues par intégration sur les fibres},
  volume = {68},
  ISSN = {1432-1297},
  url = {http://dx.doi.org/10.1007/BF01394271},
  DOI = {10.1007/bf01394271},
  number = {1},
  journal = {Inventiones Mathematicae},
  publisher = {Springer Science and Business Media LLC},
  author = {Barlet,  Daniel},
  year = {1982},
  month = feb,
  pages = {129--174}
}

@article{Bost1990,
  title = {Green’s currents and height pairing on complex tori},
  volume = {61},
  ISSN = {0012-7094},
  url = {http://dx.doi.org/10.1215/S0012-7094-90-06134-4},
  DOI = {10.1215/s0012-7094-90-06134-4},
  number = {3},
  journal = {Duke Mathematical Journal},
  publisher = {Duke University Press},
  author = {Bost,  Jean-Benoît},
  year = {1990},
  month = dec 
}

@book{GT-PDE,
  title = {Elliptic {Partial Differential Equations of Second Order}},
  ISBN = {9783642617980},
  ISSN = {2512-5257},
  url = {http://dx.doi.org/10.1007/978-3-642-61798-0},
  DOI = {10.1007/978-3-642-61798-0},
  journal = {Classics in Mathematics},
  publisher = {Springer Berlin Heidelberg},
  author = {Gilbarg,  David and Trudinger,  Neil S.},
  year = {2001}
}

@incollection{demailly1996,
  author    = {Demailly, Jean-Pierre},
  title     = {{$L^2$} {Hodge} theory and vanishing theorems},
  booktitle = {Introduction to {Hodge} Theory},
  series    = {SMF/AMS Texts and Monographs},
  volume    = {8},
  publisher = {American Mathematical Society},
  address   = {Providence, RI},
  year      = {2002},
  pages     = {1--95},
  note      = {\url{https://www-fourier.univ-grenoble-alpes.fr/~demailly/manuscripts/hodge-ams.pdf}}
}

@misc{Dai-Yoshikawa2025,
  author = {Dai, Xianzhe and Yoshikawa, Ken-Ichi},
  title = {{Degeneration of Riemann surfaces and small eigenvalues of the Laplacian}},
  year = {2025},
  note = {Preprint. \href{https://arxiv.org/abs/2509.06151v2}{arXiv:2509.06151v2}}
}

@article{Yoshikawa1997,
  title = {Degeneration of algebraic manifolds and the spectrum of {Laplacian}},
  volume = {146},
  ISSN = {2152-6842},
  url = {http://dx.doi.org/10.1017/S002776300000622X},
  DOI = {10.1017/s002776300000622x},
  journal = {Nagoya Mathematical Journal},
  publisher = {Cambridge University Press (CUP)},
  author = {Yoshikawa,  Ken-Ichi},
  year = {1997},
  month = jun,
  pages = {83--129}
}

@book{Huybrechts_K3_Book,
  title = {Lectures on {K3 Surfaces}},
  ISBN = {9781316594193},
  url = {http://dx.doi.org/10.1017/CBO9781316594193},
  DOI = {10.1017/cbo9781316594193},
  publisher = {Cambridge University Press},
  author = {Huybrechts,  Daniel},
  year = {2016},
  month = sep 
}

@article{Cheng_Li1981,
  title = {Heat kernel estimates and lower bound of eigenvalues},
  volume = {56},
  ISSN = {1420-8946},
  url = {http://dx.doi.org/10.1007/BF02566216},
  DOI = {10.1007/bf02566216},
  number = {1},
  journal = {Commentarii Mathematici Helvetici},
  publisher = {European Mathematical Society - EMS - Publishing House GmbH},
  author = {Cheng,  Shiu-Yuen and Li,  Peter},
  year = {1981},
  month = dec,
  pages = {327--338}
}

@book{Siu_Lectures_On_KE,
  title = {Lectures on {Hermitian-Einstein} Metrics for Stable Bundles and {K\"{a}hler-Einstein} Metrics},
  ISBN = {9783034874861},
  url = {http://dx.doi.org/10.1007/978-3-0348-7486-1},
  DOI = {10.1007/978-3-0348-7486-1},
  publisher = {Birkh\"{a}user Basel},
  author = {Siu,  Yum-Tong},
  year = {1987}
}

@article{Li_Tian_Heat_Kernel_1995,
  title = {On the heat kernel of the {Bergmann} metric on algebraic varieties},
  volume = {8},
  ISSN = {1088-6834},
  url = {http://dx.doi.org/10.1090/S0894-0347-1995-1320155-0},
  DOI = {10.1090/s0894-0347-1995-1320155-0},
  number = {4},
  journal = {Journal of the American Mathematical Society},
  publisher = {American Mathematical Society (AMS)},
  author = {Li,  Peter and Tian,  Gang},
  year = {1995},
  pages = {857--877}
}

@article{Filip-Tosatti21,
  title = {Canonical currents and heights for {K3} surfaces},
  volume = {11},
  ISSN = {2168-0949},
  url = {http://dx.doi.org/10.4310/CJM.2023.v11.n3.a2},
  DOI = {10.4310/cjm.2023.v11.n3.a2},
  number = {3},
  journal = {Cambridge Journal of Mathematics},
  publisher = {International Press of Boston},
  author = {Filip,  Simion and Tosatti,  Valentino},
  year = {2023},
  pages = {699--794}
}

\bigskip
\footnotesize

\textsc{Courant Institute of Mathematical Sciences, 
New York University, 251
Mercer St, New York, NY 10012
}
\par\nopagebreak
\textit{Email address}, \texttt{junyu.cao@nyu.edu
}\par\nopagebreak
\textit{Homepage}, \url{https://junyucao1024.github.io}.

\end{document}